\documentclass[11pt,twoside,reqno]{amsart}
\usepackage[utf8]{inputenc}
\usepackage[T1]{fontenc}
\usepackage{import}
\usepackage{newclude}
\usepackage{amsmath,amssymb,amsthm,mathtools}
\usepackage[toc,page]{appendix}
\usepackage{longtable,booktabs,array}
\usepackage{anysize}
\usepackage{amscd}
\usepackage{stmaryrd}
\usepackage{wasysym}
\usepackage{mathrsfs}
\usepackage[scr=boondox]{mathalfa}
\usepackage{enumitem}
\usepackage{accents}
\usepackage{dsfont}
\usepackage{soul}
\usepackage{color}
\usepackage[all]{xy}
\usepackage{float}
\usepackage{bm}
\usepackage{makecell}
\usepackage{thmtools}      
\usepackage{setspace}
\usepackage{comment}
\usepackage{tikz}
\usepackage{tikz-cd}
\usepackage{mathdots}
\usepackage{mathtools}
\usepackage{graphicx}
\usepackage[myheadings]{fullpage}
\usepackage{bbm}

\usepackage[hypertexnames=false,colorlinks,linkcolor=black,citecolor=black]{hyperref}
\usepackage[capitalize,nameinlink]{cleveref}

\crefname{appendix}{Appendix}{Appendices}
\Crefname{appendix}{Appendix}{Appendices}

\numberwithin{equation}{section}
\newcommand\mtop{.95in}
\newcommand\mbottom{.95in}
\newcommand\mleft{1in}
\newcommand\mright{1in}
\usepackage[top = \mtop, bottom = \mbottom, left = \mleft, right=\mright]{geometry}

\DeclareMathOperator{\Mat}{Mat}

\newtheorem{thm}{Theorem}[section]

\newtheorem{prop}[thm]{Proposition}

\newtheorem{lemma}[thm]{Lemma}
\newtheorem{conj}[thm]{Conjecture}

\theoremstyle{definition}
\newtheorem{defi}[thm]{Definition}

\newtheorem{rmk}[thm]{Remark}

\newcommand\reallywidehat[1]{%
\savestack{\tmpbox}{\stretchto{%
  \scaleto{%
    \scalerel*[\widthof{\ensuremath{#1}}]{\kern-.6pt\bigwedge\kern-.6pt}%
    {\rule[-\textheight/2]{1ex}{\textheight}}
  }{\textheight}%
}{0.5ex}}%
\stackon[1pt]{#1}{\tmpbox}%
}
\DeclareSymbolFont{bbold}{U}{bbold}{m}{n}
\DeclareSymbolFontAlphabet{\mathbbold}{bbold}

\makeatletter
\def\@tocline#1#2#3#4#5#6#7{\relax
  \ifnum #1>\c@tocdepth 
  \else
    \par \addpenalty\@secpenalty\addvspace{#2}%
    \begingroup \hyphenpenalty\@M
    \@ifempty{#4}{%
      \@tempdima\csname r@tocindent\number#1\endcsname\relax
    }{%
      \@tempdima#4\relax
    }%
    \parindent\z@ \leftskip#3\relax \advance\leftskip\@tempdima\relax
    \rightskip\@pnumwidth plus4em \parfillskip-\@pnumwidth
    #5\leavevmode\hskip-\@tempdima
      \ifcase #1
       \or\or \hskip 1em \or \hskip 2em \else \hskip 3em \fi%
      #6\nobreak\relax
    \hfill\hbox to\@pnumwidth{\@tocpagenum{#7}}\par
    \nobreak
    \endgroup
  \fi}
\makeatother

\newcommand{\indic}{\mathbbm{1}}
\newcommand{\Z}{\mathbb{Z}}

\newcommand{\A}{\mathfrak{A}}
\renewcommand{\S}{\mathfrak{S}}

\newcommand{\C}{\mathbb{C}}

\newcommand{\E}{\mathbb{E}}

\newcommand{\mc}{\mathcal}

\renewcommand{\l}{\lambda}

\renewcommand{\L}{\mathcal{L}}

\DeclareMathOperator{\Tr}{Tr}

\DeclareMathOperator{\Id}{Id}

\DeclareMathOperator{\sgn}{sgn}

\DeclareMathOperator{\Pois}{Pois}

\title{Hitting time mixing for random $k$-cycles}
\author{Chen Shang, Jiahe Shen, Jiyue Zeng, and Xinyi Zhang}

\date{\today}

\begin{document}

\thanks{We thank Mehtaab Sawhney for suggesting this project, and Evita Nestoridi for helpful discussions. Jiahe Shen, Jiyue Zeng, and Xinyi Zhang acknowledge support from Ivan Corwin’s NSF grant DMS-2246576 and Simons Investigator grant 929852.}

\maketitle

\begin{abstract}
In this paper, we study the random walk on the symmetric group $\S_n$ generated by the conjugacy class of $k$-cycles, where $2\le k=o(n/(\log n)^4)$. We prove that the walk exhibits hitting-time mixing: at the first time when every card has been touched, the distribution is already close to equilibrium. For odd \(k\), the equilibrium
measure is the uniform measure on \(\A_n\). For even \(k\), the walk first mixes to the parity mixture determined by the hitting time, and in our range this mixture is asymptotically \(U_{\S_n}\). Our argument combines a refined fixed-time approximation for the random $k$-cycle walk near the cutoff window with an auxiliary marking scheme inspired by Jain-Sawhney's work \cite{jain2024hitting} on random
transpositions. The main new feature is a parity-compatible coupling which handles both odd and even $k$-cycles in a unified framework. We also prove a hitting-time mixing result in the opposite regime $k\ge n-o(n^{1/2})$, and formulate a conjecture for all $2\le k\le n-1$.
\end{abstract}



\tableofcontents

\section{Introduction}\label{sec: Intro}

\subsection{Main results}

This paper studies the hitting-time mixing problem for the random \(k\)-cycle walk on \(\mathfrak S_n\). Starting from the identity, at each step we multiply by a uniformly chosen \(k\)-cycle, and we stop the walk at the first time \(\tau\) when every card has been touched. Our main result shows that, in the regime
$$
2\leq k=o\left(\frac{n}{(\log n)^4}\right),
$$
this natural stopping time is already a mixing time: \(X_\tau\) is close in total variation to the relevant uniform measure, subject only to the necessary parity constraint. This extends the hitting-time mixing theorem of Jain-Sawhney \cite{jain2024hitting} for random transpositions to random \(k\)-cycles in a range where \(k\) is allowed to grow with \(n\).

The starting point for this problem is the classical theory of mixing for random walks generated by conjugacy classes of \(\mathfrak S_n\). For random transpositions, Diaconis-Shahshahani \cite{diaconis1981generating} introduced the nonabelian Fourier-analytic method and proved sharp cutoff at time \((n\log n)/2\). Subsequent works extended and refined this circle of ideas for more general conjugacy-class walks, including character bounds, cutoff, and limiting profiles for cycle shuffles; see, for example, Roichman \cite{Roichman1996}, Lulov-Pak \cite{lulov2002rapidly}, Hough \cite{hough2016random}, Teyssier \cite{teyssier2020limit}, and Nestoridi-Olesker-Taylor \cite{nestoridi2022limit}. A more refined question asks whether the walk is already mixed at the first time when every card has been touched. This is the earliest possible time at which total-variation mixing can occur, since before this time at least one card has never been moved. For random transpositions, this hitting-time mixing problem was raised by Berestycki and formulated explicitly by Teyssier \cite[Conjecture 1.2]{teyssier2020limit}; it was recently proved by Jain-Sawhney \cite{jain2024hitting}. Their argument combines a precise fixed-time approximation of the walk with a strong-stationary-time type coupling after the cutoff window, and they noted that the method should extend beyond transpositions, in particular to random \(k\)-cycles for fixed \(k\). The present paper carries this out for \(k\)-cycles, and proves hitting-time mixing throughout the growing range above.

Let us now define the random $k$-cycle walk precisely. Let $P_n$ be the
probability measure on $\mathfrak S_n$ which is uniform on the conjugacy class
of $k$-cycles:
\begin{equation}\label{eq: measure of k-cycle}
P_n(\sigma)=
\begin{cases}
\frac{k}{n(n-1)\cdots(n-k+1)}, & \text{if } \sigma \text{ is a } k\text{-cycle},\\
0, & \text{otherwise}.
\end{cases}
\end{equation}
We consider the discrete-time Markov chain
$$
X_{t+1}=\sigma_{t+1}\cdot X_t,
$$
where $\sigma_1,\sigma_2,\ldots$ are independent random variables with law
$P_n$. Thus $X_t\sim P_n^{*t}$.

Because a $k$-cycle has sign $(-1)^{k-1}$, the limiting uniform measure has
to be interpreted with the appropriate parity constraint. Following the set-up
used by Nestoridi and Olesker-Taylor \cite{nestoridi2022limit} and Hough
\cite{hough2016random}, we compare the walk with the uniform measure on the
parity class on which it is supported: if $k$ is odd, then the walk is
supported on $\mathfrak A_n$; if $k$ is even, then the walk is supported on
$\mathfrak A_n$ at even times and on $\mathfrak A_n^c$ at odd times. We
denote by $U_{\mathfrak S_n}$ the uniform measure on $\mathfrak S_n$, and by
$U_{\mathfrak A_n}$ and $U_{\mathfrak A_n^c}$ the uniform measures on
$\mathfrak A_n$ and $\mathfrak A_n^c$, respectively.

The stopping time of interest is the first time at which every card has been
touched by the walk. More explicitly, write the $l$-th chosen $k$-cycle as
$(i_{l,1}\,\cdots\, i_{l,k})$. We define
\begin{equation}\label{eq: hitting time}
\tau
:=
\min\left\{
t\ge 1:
\min_{1\le j\le n}
\sum_{l=1}^t
\mathbf 1\{j\in\{i_{l,1},\ldots,i_{l,k}\}\}
\ge 1
\right\}.
\end{equation}
Before time $\tau$, at least one card has not been touched, so the walk cannot
be close in total variation to the relevant uniform measure. Our main theorem
shows that this obstruction is the only one in the regime
$2\le k=o(n/(\log n)^4)$: throughout this range, which goes well beyond the
fixed $k$ setting suggested by Jain-Sawhney, the walk is already mixed at the
hitting time. 

\begin{thm}\label{thm: hitting time mixing_main thm}
Suppose $n=\omega(1),2\le k\le o(n/(\log n)^{4})$. With notation as above, we have 
$$d_{TV}(X_\tau,U_{\A_n})\le \exp(-(\log\log n)^{1/2+o(1)})$$
when $k$ is odd, and
$$d_{TV}(X_\tau,U_{\mathfrak{S}_n})\le\exp(-(\log\log n)^{1/2+o(1)})$$
when $k$ is even. 
\end{thm}

The phenomenon captured by \Cref{thm: hitting time mixing_main thm} is part of a broader theme in the study of Markov-chain mixing: in many natural chains, the cutoff time is governed by the disappearance of an explicit local obstruction. This viewpoint goes back at least to the theory of strong stationary times of Aldous-Diaconis \cite{aldous1986shuffling} and Diaconis-Fill \cite{diaconis1990strong}, where one often proves mixing by identifying a stopping time at which all initially visible information has been randomized. Our random $k$-cycle hitting time problem fits naturally into this circle of ideas.

We now give an overview of the proof of \Cref{thm: hitting time mixing_main thm}. As in Jain-Sawhney's treatment of
random transpositions, we begin by proving a fixed-time approximation for the
walk near the cutoff window. The guiding heuristic is that, at time
$\frac{n}{k}(\log n+c)$, \footnote{Throughout the paper, time is always integer-valued. Whenever an expression for time is not an integer, we implicitly take its integer part and suppress the floor notation for notational convenience.} the only visible obstruction to uniformity is the set of cards which have not yet been touched. This set should be approximately Poisson, and conditional on this set, the remaining cards should be almost uniformly shuffled on the appropriate parity class. The following definition introduces the auxiliary measure $\nu_t$ which makes this approximation precise.

\begin{defi}\label{defi: construction of nu}
Suppose we have $t\in\Z$. Set $t':=t-\frac{n\log n}{k}$, and $\gamma_{n,t}:=e^{-kt'/n}$. Let $\nu_t$ denote the measure on $\mathfrak{S}_n$ defined by the following sampling process: first, sample $M_t\in\{0,1,\ldots,n\}$ according to the distribution $\mathbf{P}[M_t=x]=\mathbf{P}(\Pois(\gamma_{n,t})=x)/\mathbf{P}(\Pois(\gamma_{n,t})\le n)$, then sample a uniformly random subset $S_t$ of size $M_t$ in $[n]$. Finally,
\begin{enumerate}
\item When $k$ is odd, sample a uniformly random element of $\mathfrak{A}_{[n]\backslash S_t}$ and view it as an element of $\mathfrak{A}_n$ by fixing all of the elements in $S_t$.
\item When $k$ is even, and $t$ is odd, sample a uniformly random element of $\mathfrak{A}_{[n]\backslash S_t}^c$ and view it as an element of $\mathfrak{A}_n^c$ by fixing all of the elements in $S_t$.
\item When $k$ is even, and $t$ is even, sample a uniformly random element of $\mathfrak{A}_{[n]\backslash S_t}$ and view it as an element of $\mathfrak{A}_n$ by fixing all of the elements in $S_t$.
\end{enumerate}
\end{defi}

The measure $\nu_t$ is designed to isolate the contribution of the untouched
cards. It first chooses the number of untouched cards according to the expected
Poisson law at time $t$, then chooses their locations uniformly, and finally
puts a uniform permutation on the remaining cards, subject only to the necessary
parity constraint. Thus $\nu_t$ is an explicitly tractable proxy for the law of
$X_t$ near the cutoff window.

Our main fixed-time input is that this proxy indeed approximates the random
$k$-cycle walk in total variation. The following theorem refines the second author's result \cite[Theorem 1.6]{shen2026k}: by restricting to the smaller range of $k$ considered here, we obtain the stronger error estimate required for the hitting-time argument.

\begin{thm}[\Cref{thm: restate of X and nu} in the text]\label{thm: X and nu}
Suppose that \(2\le k=o(n/(\log n)^4)\), and that
$t' := t-\frac{n\log n}{k}$ satisfies
$|t'|\le \frac{n}{2k}\log\log\log n$. Then, we have
$$d_{TV}(X_t,\nu_t)=o\left((\log n)^{-5/2}\right).$$
\end{thm}

There is currently no known characterization of a strong stationary time for
$k$-cycle shuffling on $n$ cards. In particular, the time at which all
cards have been touched is not itself a strong stationary time. Nevertheless,
one expects that at this time the distribution of the deck is already close to
uniform, and that, as $n\to\infty$, the distribution should converge to the
uniform distribution in total variation distance. Our next goal is to strengthen
the preceding theorem into a hitting-time mixing result. To achieve this, we introduce a marking scheme starting from a time
\begin{equation}\label{eq: starting time of marking scheme}
    t^*= \frac{n\log n}{k}
    -\frac{n}{2k}\log\log\log n+\frac{4n}{k}\log\log\log\log n .
\end{equation}
By the choice of $t^*$, the set of cards not yet touched by time $t^*$ is small with high probability. We mark all cards that have already been touched, and then continue the process with a carefully chosen rule that marks the remaining
cards one by one. The marking rule is designed so that the time needed to mark all cards is close to the time needed to touch all cards.

The passage from fixed-time to hitting time follows the philosophy introduced by Jain-Sawhney for the random transposition walk. Their key insight is that a sufficiently sharp fixed-time approximation can be bootstrapped into an
``approximate sufficient statistic'' statement: near the cutoff window, once one conditions on the set of untouched cards, the remaining permutation is already close to uniform. Starting from this conditional uniformity, one can then run a
short marking procedure which mimics a strong stationary time only for the few cards that remain unmarked.

We follow this strategy, but the random $k$-cycle walk introduces new difficulties. First, a single update may involve many unmarked cards, so the
marking rule must distinguish the useful event that exactly one new card is absorbed from the exceptional event that several unmarked cards appear in the same $k$-cycle. Second, the parity of a $k$-cycle depends on $k$. When $k$ is
odd, the walk stays in $\A_n$, while when $k$ is even it alternates between $\A_n$ and $\A_n^c$. Thus the artificial marking procedure cannot simply insert idle moves, as in the transposition case; it must preserve the correct parity at
each time.

The main novelty is a parity-compatible marking scheme which resolves these two issues simultaneously. Heuristically, the scheme adds artificial one-card marking moves so that each remaining unmarked card is discovered at the correct rate, while keeping the conditional law on the marked cards uniform on the appropriate parity class. In the even-$k$ case this requires parity-correcting transpositions supported entirely on the marked set. These transpositions change
the parity without affecting which unmarked card is being marked. The proof then compares this artificial process with the genuine $k$-cycle walk and shows that the accumulated discrepancy before all cards are marked is negligible.

\subsection{Future directions.}

The range $2\leq k=o(n/(\log n)^4)$ in
\Cref{thm: hitting time mixing_main thm} should not be interpreted as the
conjectural endpoint of hitting-time mixing for random $k$-cycles. Rather,
this is the regime in which our fixed-time approximation theorem is strong
enough to be combined with the Jain-Sawhney hitting-time argument. We expect
that the same qualitative phenomenon should persist for much larger values of
$k$: once every card has been touched, there should be no further obstruction
to mixing, apart from the necessary parity constraint. The following theorem provides
evidence for this expectation in the opposite, large-cycle regime.

\begin{thm}\label{thm: large cycle}
Suppose $n=\omega(1)$, $k\leq n-1$, and $k\geq n-o(n^{1/2})$. Then, we have
$$
d_{TV}(X_\tau,U_{\mathfrak A_n})=o(1).
$$
\end{thm}

Indeed, as we will see in \Cref{sec: large cycles}, \Cref{thm: large cycle} corresponds to the situation where two cycles
are enough to complete the hitting with high probability. After the first
$k$-cycle, only $n-k=o(n^{1/2})$ cards remain untouched, and a second
independent $k$-cycle covers all of them with probability tending to one. Thus
\Cref{thm: large cycle} shows that, in this regime, the walk is already mixed at the moment
when the second cycle completes the hitting of all cards.

The endpoint $k=n$, however, is degenerate and must be excluded. In that case
the walk hits all cards after a single step, but the distribution at that time
is supported on the conjugacy class of $n$-cycles, and hence is certainly not
close to uniform. With this exception in mind, it is natural to formulate the
following conjecture.

\begin{conj}\label{conj: hitting time mixing all k}
Suppose $n=\omega(1)$ and $2\leq k\leq n-1$. Then the following hold. If $k$ is odd, then
$$
d_{TV}(X_\tau,U_{\mathfrak A_n})=o(1).
$$
If $k$ is even, then 
$$
d_{TV}\!\left(X_\tau,\,\mathbf P(\tau \text{ is even})\,U_{\mathfrak A_n}+\mathbf P(\tau \text{ is odd})\,U_{\mathfrak A_n^c}
\right)=o(1).
$$
\end{conj}

Let us explain the form of the conjecture in the even-\(k\) case. In fact, in \Cref{thm: hitting time mixing_main thm}, when \(k\) is even, we first prove the stronger intermediate statement that \(X_\tau\) is close to the parity mixture
$$
\mathbf P(\tau \text{ is even})\,U_{\mathfrak A_n}+
\mathbf P(\tau \text{ is odd})\,U_{\mathfrak A_n^c}.
$$
We then show, in the regime \(k=o(n/(\log n)^4)\), that the hitting time is asymptotically equally likely to be even or odd, which allows us to replace this mixture by \(U_{\mathfrak S_n}\). For general \(k\), however, the assertion that \(\tau\) has asymptotically balanced parity is no longer expected to hold. Thus the natural even-\(k\) conjecture is convergence to the parity mixture determined by the actual distribution of \(\tau\), rather than necessarily to \(U_{\mathfrak S_n}\).

It would also be interesting to study hitting-time mixing for random walks
generated by more general class measures on $\mathfrak S_n$. In the present
paper the driving measure is supported on a single conjugacy class, namely the
class of $k$-cycles, but the Fourier-analytic framework applies equally well
to any class measure. The main difficulty is then to understand which
irreducible representations make the dominant contribution near the hitting
time, and how the touching process interacts with the representation-theoretic
profile of the chosen class measure. Recent work of Olesker-Taylor, Teyssier,
and Th\'evenin \cite{olesker2025sharp} develops sharp character bounds
and cutoff results for broad families of conjugacy-class walks on
$\mathfrak S_n$, and provides useful representation-theoretic background for
such a direction.


\subsection{Notation}

We write \([n]:=\{1,\ldots,n\}\). For a subset \(S\subseteq[n]\), we write \(\mathfrak S_S\) for the group of permutations of \(S\). In particular, when \(S=[n]\setminus T\), we also write \(\mathfrak S_{[n]\setminus T}\) and regard it naturally as a subgroup of \(\mathfrak S_n\) by fixing every point of \(T\). The same convention applies to \(\mathfrak A_S\) and \(\mathfrak A_S^c\).

For a probability measure \(\mu\), we write \(X\sim\mu\) to mean that \(X\) has law \(\mu\), and we write \(\mathcal L(X)\) for the law of a random variable \(X\). The total variation distance between two probability measures \(\mu\) and \(\nu\) on a finite set is denoted by
$$
d_{\mathrm{TV}}(\mu,\nu)
:=
\frac12\sum_x |\mu(x)-\nu(x)|.
$$
When no confusion can arise, we also write \(d_{\mathrm{TV}}(X,Y)\) for \(d_{\mathrm{TV}}(\mathcal L(X),\mathcal L(Y))\). We use \(\mathbf P\) and $\E$ for probability and expectation.

For a permutation \(\sigma\in\mathfrak S_n\), we write
$$
\operatorname{Fix}(\sigma):=\{i\in[n]:\sigma(i)=i\}
$$
for its set of fixed points. If \(a\) is a nonnegative integer and \(b\geq 0\), we use the falling factorial notation
$$
(a)_b:=a(a-1)\cdots(a-b+1),
$$
with the convention \((a)_0=1\).

We write \(f=O(g)\) if \(|f|\le C|g|\) for some absolute constant \(C>0\), \(f=o(g)\) if \(f/g\to0\), and \(f=\omega(g)\) if \(f/g\to+\infty\). All asymptotic notation is understood in absolute value. Moreover, the implicit constants are taken along the given sequence of positive integer inputs; in particular, they are not intended uniformly over all such inputs, but only along the sequence under consideration.

\subsection{Outline of the paper} 

The remainder of the paper is organized as follows. In \Cref{sec: Jain-Sawhney}, we prove the fixed-time approximation theorem, \Cref{thm: X and nu}. In \Cref{sec: fixed to hitting}, we use this approximation to prove the hitting-time mixing result, \Cref{thm: hitting time mixing_main thm}. Finally, in \Cref{sec: large cycles}, we prove \Cref{thm: large cycle} for the regime where \(k\) is close to \(n\).

\section{Approximating the distribution at a fixed time}\label{sec: Jain-Sawhney}

The goal of this section is to prove the following theorem, which provides a detailed reformulation of \Cref{thm: X and nu}.

\begin{thm}\label{thm: restate of X and nu}
Let $n_1,n_2,\ldots$,$k_1,k_2,\ldots$, and $t_1,t_2,\ldots$ be sequences of positive integers such that:
\begin{enumerate}
\item $\lim_{N\rightarrow\infty}n_N=\infty$.
\item $2\le k_N\le n_N$ for all $N\ge 1$, and $\lim_{N\rightarrow\infty}\frac{k_N}{n_N/(\log n_N)^4}=0$.
\item $|t_N-\frac{n_N\log n_N}{k_N}|\le \frac{n_N}{2k_N}\log\log\log n_N$ for all $N\ge 1$. 
\end{enumerate}
Define the measures $X_{t_N},\nu_{t_N}$ over $\mathfrak{S}_{n_N}$ the same way as in \Cref{sec: Intro}. Then, we have 
$$\lim_{N\rightarrow\infty}d_{TV}(X_{t_N},\nu_{t_N})\cdot(\log n_N)^{5/2}=0.$$
\end{thm}

For the rest of the proof, for notational ease, we will drop off the subscript and simply write $k,n$ for $k_N,n_N$. Now, let us recall necessary background from representation theory.

\subsection{Nonabelian Fourier transform and related notations}

Given a finite group $G$, we let $\widehat G$ denote the set of irreducible representations. We define convolution $f_1,f_2:G\rightarrow\C$ to be
$$(f_1*f_2)(z):=\sum_{x\in G}f_1(x)f_2(x^{-1}z).$$
For $\rho\in \widehat G$, let $d_\rho$ denote the dimension of $\rho$. The nonabelian Fourier transform for a function $f:G\rightarrow\C$ is the map $\widehat f$ given by
$$\widehat f(\rho):=\sum_{g\in G}f(g)\rho(g)\in\Mat_{d_\rho}(\C),\quad\forall \rho\in\widehat G.$$
Then, for all $f_1,f_2:G\rightarrow\C$ and $\rho\in\widehat G$, we have
\begin{equation}\label{eq: convolution to multiplication}
\widehat{f_1*f_2}(\rho)=\widehat f_1(\rho)\cdot\widehat f_2(\rho).
\end{equation}
In this paper, we always take $G=\mathfrak{S}_n$. Following the standard construction of Specht modules, we can identify elements in $\widehat{\mathfrak{S}}_n$ with (positive integer) partitions of $n$. From now on, we use $\l\vdash n$ to denote that $\l$ is a partition of $n$. Given a partition $\l=(\l_1,\ldots,\l_j)$ with $\l_1\ge\cdots\ge\l_j$, we let $\l'$ denote the conjugate partition and $\l^*=(\l_2,\ldots,\l_j)$ be the partition of $n-\l_1$ obtained by truncating $\l$. Denote by $l(\lambda)$ the number of parts of $\lambda$, and
$$
\lambda^\bullet:=\bigl((\lambda')^*\bigr)'\,,
$$
which is obtained from $\lambda$ by deleting its first column. 

Let
\begin{equation}\label{eq: long first row and column partitions}
\mathcal L \ :=\ \Bigl\{\lambda\vdash n:\ \lambda_1\ge n-\log n\ \ \text{or}\ \ l(\l)\ge n-\log n\Bigr\}
\end{equation}
be the set of partitions with long first row or long first column. It is clear that $\mathcal L=\mathcal L_1\bigsqcup \mathcal L_2$, where
$$\mathcal L_1:=\ \Bigl\{\lambda\vdash n:\ \lambda_1\ge n-\log n\Bigr\},\quad\mathcal L_2:=\ \Bigl\{\lambda\vdash n:\ l(\l)\ge n-\log n\Bigr\}.$$
Moreover, under the conjugation map, the partitions in $\mathcal{L}_1$ have one-to-one correspondence with the partitions in $\mathcal{L}_2$.  

Probability measures on $\mathfrak{S}_n$ will be viewed as real-value functions. In our applications, we will furthermore restrict attention to functions $f$ which are class functions, i.e., functions that are constants on conjugacy classes. 
\begin{lemma}[Schur's lemma]
Suppose $f:\mathfrak{S}_n\rightarrow \C$ is a class function. Then, we have
$$\widehat f(\lambda)=\frac{\sum_{g\in G}f(g)\chi_\lambda(g)}{d_\lambda}\Id_{d_\lambda},$$
where $\chi_\lambda$ is the character corresponding to $\lambda\vdash n$.
\end{lemma}

In particular, let $P_n$ be the probability measure over $\mathfrak{S}_n$ as in \eqref{eq: measure of k-cycle}. Then, its Fourier transform is given by
\begin{equation}\label{eq: Fourier transform of Pn}
\widehat P_n(\lambda)=\frac{\chi_\lambda(\tau_k)}{d_\lambda}\Id_{d_\lambda},\quad \lambda\vdash n.
\end{equation}
where $\tau_k$ is an arbitrary $k$-cycle. Following \cite{nestoridi2022limit}, we will also use the abbreviation $s_\lambda(k)$ for $\frac{\chi_\lambda(\tau_k)}{d_\lambda}$. In particular, by the hook length formula, we have $d_\l=d_{\l'}$. Furthermore, the natural isomorphism $S^{\lambda'}\cong S^\lambda\otimes\sgn$ leads to the fact that $\chi_{\l'}(k)=\chi_\l(k)\cdot(-1)^{k-1}$, and therefore
$$s_{\l'}(k)=s_\l(k)\cdot(-1)^{k-1}.$$

\subsection{Character estimates for $P_n^{*t}$}

Following \eqref{eq: convolution to multiplication} and \eqref{eq: Fourier transform of Pn}, it is clear that
$$\widehat {P_n^{*t}}(\l)=s_\l(k)^t\Id_{d_\l},\quad\l\vdash n.$$
In the following, we provide quantitative bounds that will be useful later. Our range of $t$ will always follow \Cref{thm: X and nu}.

\begin{lemma}[{\cite[Lemma 3.3]{jain2024hitting}}]\label{lem: dimension bound}
Let $\l\vdash n$, and $r:=n-\l_1$. Then, we have
\begin{enumerate}
\item $d_\lambda\le\binom{n}{r}d_{\lambda^*}\le\binom{n}{r}\cdot \sqrt{r!}\le n^r/\sqrt{r!}.$
\item When $r\le \log n$, we have
$$d_\l=\binom{n}{r}\cdot d_{\l^*}\left(1-\frac{r}{n}+O(n^{-2+o(1)})\right).$$
\end{enumerate}
\end{lemma}

\begin{prop}[{\cite[Proposition 2.6]{shen2026k}}]\label{prop: small r}
Suppose $r:=n-\l_1\in[1, \log n]$, and $2\le k\le o(n/(\log n)^4)$. Moreover, suppose that $t':=t-\frac{n\log n}{k}$ satisfies $|t'|\le \frac{n\log\log\log n}{2k}$. Then, we have
$$s_\l(k)^t=\left(1-O\left(\frac{r(k+r)\log n}{n}\right)\right)\exp\left(-\frac{rkt}{n}\right).$$
\end{prop}

\begin{proof}
Following \cite[Proposition 2.6]{shen2026k}, we already have
$$s_\l(k)^t=\exp\left(-\frac{rkt}{n}-\frac{r(k+r)\log n}{2n}(1+o(1)+O(1/k))\right).$$
Now, since $k=o(n/(\log n)^4)$ and $r\le \log n$, we have $\frac{r(k+r)\log n}{2n}(1+o(1)+O(1/k))=o(1)$, so that
$$s_\l(k)^t=\left(1-O\left(\frac{r(k+r)\log n}{n}\right)\right)\exp\left(-\frac{rkt}{n}\right).$$
\end{proof}

The following bound generalizes \cite[Lemma 3.4]{jain2024hitting}. It shows that the sum $\sum_{\lambda\vdash n}d_\l^2|s_\l(k)|^{2t}$ is mainly contributed by partitions in $\mathcal{L}$. 

\begin{thm}[{\cite[Theorem 2.15]{shen2026k}}]\label{thm: removal}
Suppose that $2\le k\le o(n/(\log n)^4)$, and that $t':=t-\frac{n\log n}{k}$ satisfies $|t'|\le \frac{n\log\log\log n}{2k}$. Then, we have
$$\sum_{\substack{\lambda\vdash n\\\lambda\notin \mc{L}}}d_\l^2|s_\l(k)|^{2t}=n^{-\omega(1)}.$$
\end{thm}

Let us turn to character estimates for $\nu_t$. Fix integers $n\ge 1$ and $0\le M\le n/3$. In this case, for $\l\vdash n$, at most one of the following holds:
\begin{enumerate}
\item $\l_1\ge n-M$.
\item $l(\l)\ge n-M$.
\end{enumerate}

Define a probability measure $\xi_M$ (resp. $\xi_M^c$) on $\mathfrak{A}_n$ as follows: choose a subset $S\subset[n]$ uniformly among all subsets of size $M$, then sample a permutation $\pi$ uniformly from $\A_{[n]\setminus S}$ (resp. $\A_{[n]\setminus S}^c$ ), and extend it to an element of $\mathfrak{A}_n$ (resp. $\mathfrak{A}_n^c$) by fixing every point of $S$. The following lemma gives expression of the Fourier transform of $\xi_M$ and $\xi_M^c$.

\begin{thm}[Fourier transform of $\xi_M$ and $\xi_M^c$]
\label{thm :zetaM_Fourier_An}
We have
$$
\widehat{\xi_M}(\lambda)=\frac{K_{\lambda,\mu}+K_{\lambda',\mu}}{d_\lambda}\,\Id_{d_\lambda}, \quad\widehat{\xi_M^c}(\lambda)=\frac{K_{\lambda,\mu}-K_{\lambda',\mu}}{d_\lambda}\,\Id_{d_\lambda},
$$
where
$$
\mu=\mu_M:=(n-M,1,1,\dots,1)\vdash n,
$$
and $K_{\alpha,\mu}$ denotes the Kostka number.

Moreover, if $\l_1<n-M$ and $l(\l)<n-M$, then $\widehat{\xi_M}(\lambda)=\widehat{\xi_M^c}(\lambda)=0$. If $\l_1\ge n-M$, we have
$$\widehat{\xi_M}(\lambda)=\binom{M}{\,n-\lambda_1\,}\frac{
d_{\lambda^*}}{d_\lambda}\,\Id_{d_\lambda},\quad\widehat{\xi_M^c}(\lambda)=\binom{M}{\,n-\lambda_1\,}\frac{
d_{\lambda^*}}{d_\lambda}\,\Id_{d_\lambda}.$$
If $l(\l)\ge n-M$, we have
$$\widehat{\xi_M}(\lambda)=\binom{M}{\,n-l(\lambda)\,}\frac{d_{\lambda^\bullet}}{d_\lambda}\,\Id_{d_\lambda},\quad\widehat{\xi_M^c}(\lambda)=-\binom{M}{\,n-l(\lambda)\,}\frac{d_{\lambda^\bullet}
}{d_\lambda}\,\Id_{d_\lambda}.$$
\end{thm}

\begin{proof}
This is \cite[Theorem 2.16]{shen2026k}.
\end{proof}

We already have the necessary preparation to prove \Cref{thm: restate of X and nu}. 

\begin{proof}[Proof of \Cref{thm: restate of X and nu}]
In order to simplify computations, we consider the following truncated version $\mu_{n,k}^t$ of $\nu_{n,k}^t$, which is defined in a similar fashion, except that the distribution of $M_t$ (i.e., the size of the random subset $S_t$) is given by
$$\mathbf{P}[M_t=x]=\mathbf{P}[\Pois(\gamma_{n,t})=x]/\mathbf{P}[\Pois(\gamma_{n,t})\le\log n],\quad x\le \log n.$$
Here, $\gamma_{n,t}:=\exp(-kt'/n)$ follows from \Cref{defi: construction of nu}. Due to our assumption on $t'$, we have $\gamma_{n,t}\le\log\log n$, and therefore
\begin{equation}\label{eq: Poisson truncate}
\mathbf{P}[\Pois(\gamma_{n,t})\ge \log n]\le \mathbf{P}[\Pois(\log\log n)\ge \log n]=n^{-\omega(1)}.
\end{equation}
Thus, it follows from the natural decoupling that $d_{TV}(\nu_{n,k}^t,\mu_{n,k}^t)=n^{-\omega(1)}$. Thus, it suffices to prove the same bound for $d_{TV}(P_n^{*t},\mu_{n,k}^t)$. Let $f_1$ be the probability density function of $P_n^{*t}$, and $f_2$ be the probability density function of $\mu_{n,k}^t$. Using the Cauchy-Schwarz inequality followed by the Plancherel formula, we have that (the dagger superscript refers to the conjugate transpose matrix)
\begin{align}
\begin{split}
4\cdot d_{TV}(P_n^{*t},\mu_{n,k}^t)^2&\le n!\sum_{\sigma\in\S_n}(f_1(\sigma)-f_2(\sigma))^2\\
&=\sum_{\l\vdash n}d_\l\Tr\left((\widehat f_1(\l)-\widehat f_2(\l))(\widehat f_1(\l)-\widehat f_2(\l))^\dag\right)\\
&=\sum_{\l\notin\L}d_\l^2|s_\l(k)|^{2t}+\sum_{\l\in\L}d_\l\Tr\left((\widehat f_1(\l)-\widehat f_2(\l))(\widehat f_1(\l)-\widehat f_2(\l))^\dag\right)\\
&=n^{-\omega(1)}+\sum_{\l\in\L}d_\l\Tr\left((\widehat f_1(\l)-\widehat f_2(\l))(\widehat f_1(\l)-\widehat f_2(\l))^\dag\right).\\
\end{split}
\end{align}
Here, the third line follows from \Cref{thm :zetaM_Fourier_An} since $\widehat f_2(\l)=0$ when $\l\notin\L$, and the final line follows from \Cref{thm: removal}. Moreover, due to the symmetry under conjugation, the partitions in $\mathcal L_1$ and $\mathcal L_2$ contribute equally to the whole sum:
the sign change $s_{\lambda'}(k)^t=(-1)^{t(k-1)}s_\lambda(k)^t$ is matched
by the corresponding sign change in the Fourier coefficient of the parity-supported comparison measure. Hence, we only need to prove
\begin{equation*}
\sum_{\l\in\L_1}d_\l\Tr\left((\widehat f_1(\l)-\widehat f_2(\l))(\widehat f_1(\l)-\widehat f_2(\l))^\dag\right)=o\left((\log n)^{-5}\right).
\end{equation*}
Now, suppose that $\l\in\mathcal{L}_1$, and denote $r:=n-\l_1$. On the one hand, by \Cref{prop: small r}, we have
\begin{equation}
\widehat f_1(\l)=s_\l(k)^t\Id_{d_\l}=\left(1-O\left(\frac{r(k+r)\log n}{n}\right)\right)\exp\left(-\frac{rkt}{n}\right)\Id_{d_\l}.
\end{equation}
On the other hand, notice that
$$\mu_{n,k}^t\sim\begin{cases}
\sum_{l=0}^{\log n}\xi_l^c\cdot\frac{\mathbf{P}[\Pois(\gamma_{n,t})=l]}{\mathbf{P}[\Pois(\gamma_{n,t})\le\log n]} & t \text{ odd},k\text{ even}\\
\sum_{l=0}^{\log n}\xi_l\cdot\frac{\mathbf{P}[\Pois(\gamma_{n,t})=l]}{\mathbf{P}[\Pois(\gamma_{n,t})\le\log n]} & \text{else}
\end{cases},$$
where $\xi_l,\xi_l^c$ are the distributions in the statement of \Cref{thm :zetaM_Fourier_An}. Therefore, it follows from \Cref{thm :zetaM_Fourier_An} and \Cref{lem: dimension bound} that
\begin{align}
\begin{split}
\widehat f_2(\l)&=\mathbf{P}[\Pois(\gamma_{n,t})\le\log n]^{-1}\cdot\sum_{l=0}^{\log n}\frac{d_{\lambda^*}\binom{l}{r}}{d_\lambda}\mathbf{P}[\Pois(\gamma_{n,t})=l]\Id_{d_\l}\\
&=\left(1-\frac rn+O(n^{-2+o(1)})\right)\sum_{l=0}^{\log n}\binom{l}{r}\cdot r!n^{-r}\cdot\frac{\gamma_{n,t}^l \exp(-\gamma_{n,t})}{l !}\Id_{d_\l}\\
&=\left(1-\frac rn+O(n^{-2+o(1)})\right)\sum_{l=r}^{\log n}n^{-r}\cdot\frac{\gamma_{n,t}^l \exp(-\gamma_{n,t})}{(l-r)!}\Id_{d_\l}\\
&=\left(1-\frac rn+O(n^{-2+o(1)})\right)n^{-r}\cdot\gamma_{n,t}^r\cdot\mathbf{P}[\Pois(\gamma_{n,t})\le \log n-r]\Id_{d_\l}\\
&=\left(1-\frac rn+O(n^{-2+o(1)})\right)\exp\left(-rkt/n\right)\mathbf{P}[\Pois(\gamma_{n,t})\le\log n-r]\Id_{d_\l}.
\end{split}
\end{align}
Here, we use the fact from \eqref{eq: Poisson truncate} that $\mathbf{P}[\Pois(\gamma_{n,t})\le\log n]=1-n^{-\omega(1)}$. Now, suppose
$$\widehat f_1(\l)-\widehat f_2(\l)=\alpha_\l\Id_{d_\l},\quad \l\in\L_1,$$
so that 
$$\sum_{\l\in\L_1}d_\l\Tr\left((\widehat f_1(\l)-\widehat f_2(\l))(\widehat f_1(\l)-\widehat f_2(\l))^\dag\right)=\sum_{\l\in\L_1}d_{\l}^2|\alpha_{\l}|^2.$$
We break the above sum into two parts, depending on whether $r\ge\frac12\log n$ or not.
\begin{enumerate}
\item Suppose $r\ge \frac 12\log n$. In this case, we can use the rough bound $\widehat f_1(\l)-\widehat f_2(\l)=\alpha_\l\Id_{d_\l}$ where $|\alpha_\l|\le 3\exp(-rkt/n)$, along with the dimension bound $d_\l^2\le n^{2r}/r!$ in \Cref{lem: dimension bound} to see that
\begin{align}
\begin{split}
\sum_{n-\l_1\in[\frac 12\log n,\log n]}d_\l^2|\alpha_\l|^2&\le9\sum_{r\in[\frac 12\log n,\log n]}\sum_{\l_1=n-r}\exp(-2rkt/n)\frac{n^{2r}}{r!}\\
&\le O(\log n)\cdot\max_{r\in[\frac 12\log n,\log n]}\frac{\exp(-2rkt'/n)}{r!}\exp\left(\pi\sqrt{2r/3}\right)\\
&= n^{-\omega(1)}.
\end{split}
\end{align}
Here, in the last inequality, we used Stirling's formula and the assumption $|2kt'/n|\le\log\log\log n$.

\item Suppose $r<\frac 12\log n$. In this case, we have 
$$\mathbf{P}[\Pois(\gamma_{n,t})\le \log n-r]\ge \mathbf{P}\left[\Pois(\log\log n)\le \frac 12\log n\right]\ge1-n^{-\omega(1)}.$$ 
Therefore, we have $\widehat f_2(\l)=\left(1-\frac rn+O(n^{-2+o(1)})\right)\exp\left(-rkt/n\right)$, and
\begin{align*}
\begin{split}
\alpha_\lambda&=\exp(-rkt/n)\left(\left(1-O\left(\frac{r(k+r)\log n}{n}\right)\right)-\left(1-\frac rn+O(n^{-2+o(1)})\right)\right)\\
&=O\left(\frac{r(k+r)\log n}{n}\right)\exp(-rkt/n).
\end{split}
\end{align*}
Also, by \Cref{lem: dimension bound}, we obtain
\begin{equation}\label{eq: estimate of sum of d_l^2}
\sum_{\l_1=n-r}d_\l^2\le \binom{n}{r}^2\sum_{\mu\vdash r}d_\mu^2=\binom{n}{r}^2\cdot r!\le\frac{n^{2r}}{r!}.
\end{equation}
Thus, applying \eqref{eq: estimate of sum of d_l^2}, we have 
\begin{align}
\begin{split}
\sum_{n-\l_1<\frac 12\log n}d_\l^2|\alpha_\l|^2&\le \sum_{r<\frac 12\log n}O\left(\frac{r(k+r)\log n}{n}\right)^2\frac{n^{2r}}{r!}\exp(-2rkt/n)\\
&\le \sum_{r\ge 0}O\left(\frac{r(k+r)\log n}{n}\right)^2\frac{\exp(-2rkt'/n)}{r!}\\
&=O\left(\frac{k^2(\log n)^2}{n^2}\cdot\log n\right)\\
&=o\left((\log n)^{-5}\right).
\end{split}
\end{align}
Here, on the third line, we use the fact that $|2kt'/n|\le\log\log\log n$, which implies 
$$\sum_{r\ge 0}r^m\frac{\exp(-2rkt'/n)}{r!}\le\sum_{r\ge 0}r^m\frac{(\log\log n)^r}{r!}=O(\log n)$$
for integers $m=2,3,4$.
\end{enumerate}

Combining our discussion of $r\in[1,\frac 12\log n)$ and $r\in[\frac 12\log n,\log n]$ above, we complete the proof.
\end{proof}

\section{From Fixed Time Approximation to Hitting Time Mixing}\label{sec: fixed to hitting}

\subsection{Reduction to distributional approximation at a fixed time}
The role of this subsection is to upgrade the fixed-time approximation of \Cref{sec: Jain-Sawhney} into a conditional approximation given the exact set of untouched cards. This will provide the initial distribution for the marking argument started at $t^*$ defined in \eqref{eq: starting time of marking scheme}.

For a subset $T \subset [n]$, let $\mathcal{G}(T) = \mathcal{G}^t(T)$ denote the event that the random walk does not touch $T$ up to time $t$ and let $\mathcal{F}(T) = \mathcal{F}^t(T)$ denote the event that the random walk touches everything outside $T$ and does not touch $T$ up to time $t$. We use $\mathcal{F}(M)$ to denote $\mathcal{F}([M])$ for an integer $1\leq M \leq n$. Recall from \Cref{thm: X and nu} that $t =  \frac{n \log n}{k}  + t'$ where $|t'| \leq \frac{n}{2k}\log \log \log n$.

\begin{lemma}\label{lemma::Self-reduction lemma}
    Let $M \in \mathbb{N}$ and $T \subset [n]$ such that $M \leq |T| = o((\log n)^{3/2})$. Recall that $\gamma_{n,t} = \exp(-kt'/n).$ Then
    \begin{enumerate}
        \item $\mathbf{P}[\mathcal{G}(T)] = (1+(\log n)^{-1/2+o(1)})\mathbf{P}[\mathcal{G}([M])]\gamma_{n,t}^{|T| - M}n^{M-|T|}$,
        \item $\mathbf{P}[\mathcal{F}([M])] =(1+(\log n)^{-1/2+o(1)})\mathbf{P}[\mathcal{G}([M])]\exp(-\gamma_{n,t})$,
        \item $\sum_{T \supseteq [M], |T| \leq o((\log n)^{3/2})} \mathbf{P}[\mathcal{G}(T)] \leq (\log n)^{o(1)}\mathbf{P}[\mathcal{F}([M])]$.
    \end{enumerate}
\end{lemma}

\begin{proof}
For the first item,
\begin{equation}
\begin{aligned}
        \mathbf{P}[\mathcal{G}(T)] = \left(\frac{(n-|T|)_k}{(n)_k}\right)^{t} &= \exp\left(t\sum_{i=0}^{k-1}\log\left(1-\frac{|T|}{n-i}\right)\right)\\
        &= \exp\left(-kt\frac{|T|}{n}\right)\exp\left(O\left(\frac{tk^2|T|}{n^2}\right)+O\left(\frac{tk|T|^2}{n^2}\right)\right).
\end{aligned}
\end{equation}
Since $O\left(\frac{tk^2|T|}{n^2}\right) + O\left(\frac{tk|T|^2}{n^2}\right)=O\left(\frac{1}{(\log n)^{1/2}}\right)$, we have
\begin{equation}
\begin{aligned}
        \mathbf{P}[\mathcal{G}(T)]
        &= (1+(\log n)^{-1/2+o(1)})\exp\left(-kt\frac{|T|}{n}\right)\\
        &= (1+(\log n)^{-1/2+o(1)})\exp\left(-kt\frac{M}{n}\right)\exp\left(-kt\frac{|T|-M}{n}\right)\\
        &=  (1+(\log n)^{-1/2+o(1)})\mathbf{P}[\mathcal{G}(M)]\gamma_{n,t}^{|T|-M}n^{M-|T|}.
\end{aligned}
\end{equation}

For the second item, we use the principle of inclusion-exclusion to rewrite the indicator function of $\mathcal{F}(M)$ as a sum of indicator functions of $\mathcal{G}(T)$,
\begin{equation}
    \indic \{\mathcal{F}(M)\} = \sum_{T \supseteq [M]} (-1)^{|T| - M} \indic \{\mathcal{G}(T)\}.
\end{equation}
By the Bonferroni inequality, we have
\begin{equation}\label{eq: Bonferroni}
\sum_{T \supseteq [M], |T| \leq M + 2L +1} (-1)^{|T| - M} \indic \{\mathcal{G}(T)\} \leq \indic\{\mathcal{F}(M)\} \leq \sum_{T \supseteq [M], |T| \leq M + 2L} (-1)^{|T| - M} \indic \{\mathcal{G}(T)\}
\end{equation}
for any positive integer $L$. Take \(L=\lfloor\log n\rfloor\) and take expectations on both sides of \eqref{eq: Bonferroni}. Since \(\mathbf P[\mathcal G(T)]\) depends only on \(|T|\), the contribution of all sets \(T\supseteq[M]\) with \(|T|=M+j\) is
\(\binom{n-M}{j}\mathbf P[\mathcal G(T)]\), where \(T\) is any such set. Hence
$$
\sum_{j=0}^{2L+1}
(-1)^j
\binom{n-M}{j}
\mathbf P[\mathcal G(T)]
\leq
\mathbf P[\mathcal F(M)]
\leq
\sum_{j=0}^{2L}
(-1)^j
\binom{n-M}{j}
\mathbf P[\mathcal G(T)].
$$
Here, in the term indexed by \(j\), \(T\) denotes any set satisfying \(T\supseteq[M]\) and \(|T|=M+j\).
We then apply the first item
\[
\begin{aligned}
(1+(\log n)^{-1/2+o(1)})\frac{\mathbf{P}[\mathcal{F}(M)]}{\mathbf{P}[\mathcal{G}(M)]} &= \sum_{j=0}^{2L} \binom{n-M}{j}\frac{(-\gamma_{n,t})^j}{n^j} \pm  \binom{n-M}{2L+1}\frac{(-\gamma_{n,t})^{2L+1}}{n^{2L+1}}\\
&=(1+n^{-1+o(1)})\left(\sum_{j=0}^{2L} \frac{(-\gamma_{n,t})^j}{j!} \pm \frac{(-\gamma_{n,t})^{2L+1}}{(2L+1)!}  \right)\\
&= (1+n^{-1+o(1)})\exp(-\gamma_{n,t})
\end{aligned}
\]

where the last equality follows from the fact that $\gamma_{n,t} \leq \sqrt{\log \log n}$ and $L=\lfloor\log n\rfloor$.

For the last item, we use the first and second items
\[
\begin{aligned}
    \sum_{T \supseteq [M], T \leq o((\log n)^{3/2})} \mathbf{P}[\mathcal{G}(T)] &= \sum_{s=0}^{o((\log n)^{3/2})}\binom{n-M}{s}\mathbf{P}\left[\mathcal{G}([M+s])\right]\\
    & = (1+(\log n)^{-1/2+o(1)}) \mathbf{P}[\mathcal{F}(M)]e^{\gamma_{n,t}} \sum_{s=0}^{ o((\log n)^{3/2})} \binom{n-M}{s}\left( \frac{\gamma_{n,t}}{n}\right)^s\\
    & \leq (1+(\log n)^{-1/2+o(1)}) \mathbf{P}[\mathcal{F}(M)]e^{\gamma_{n,t}} \sum_{s=0}^{ o((\log n)^{3/2})} \binom{n}{s}\left( \frac{\gamma_{n,t}}{n}\right)^s\\
    &\leq (1+(\log n)^{-1/2+o(1)}) \mathbf{P}[\mathcal{F}(M)]e^{\gamma_{n,t}} \sum_{s=0}^{ o((\log n)^{3/2})}\left( \frac{\gamma_{n,t}^s}{s!}\right)\\
    &\leq (1+(\log n)^{-1/2+o(1)}) \mathbf{P}[\mathcal{F}(M)]e^{2\gamma_{n,t}}\\
    &\leq (\log n)^{o(1)}\mathbf{P}[\mathcal{F}(M)]
\end{aligned}
\]
where the last equality follows from the fact that $\gamma_{n,t} \leq \sqrt{\log \log n}$. 
\end{proof}

For the next lemma, suppose that $|t'| \leq \frac{n}{2k}\log \log \log n$. For a subset $T \subset [n]$, let 
\[
t_T := \frac{(n-|T|)\log(n-|T|)}{k}+ \frac{t'(n-|T|)}{n}.
\]
Then, we have 
\begin{equation}\label{eq: gamma_T}
\begin{aligned}
    \gamma_{n-|T|,t_T} = \exp(-kt'(n-|T|)/n(n-|T|)) = \exp(-kt'/n) =\gamma_{n,t}.
\end{aligned}
\end{equation}
\begin{defi}\label{def: nu^T_t}
    Let $\nu^T_t$ denote the distribution on $\mathfrak{S}_{[n]\setminus T}$  defined by the following sampling process. First, sample $M_t\in\{0,1,\ldots,n-|T|\}$ according to the distribution $\mathbf{P}[M_t=x]=\mathbf{P}(\Pois(\gamma_{n-|T|,t_T})=x)/\mathbf{P}(\Pois(\gamma_{n-|T|,t_T})\le n-|T|)$, then sample a uniformly random subset $S_t$ of size $M_t$ in $[n]\setminus T$. Finally,
\begin{enumerate}
\item When $k$ is odd, sample a uniformly random element of $\mathfrak{A}_{[n]\backslash S_t \cup T}$ and view it as an element of $\mathfrak{A}_n$ by fixing all of the elements in $S_t$ and $T$.
\item When $k$ is even, and $t$ is odd, sample a uniformly random element of $\mathfrak{A}_{[n]\backslash S_t\cup T}^c$ and view it as an element of $\mathfrak{A}_n^c$ by fixing all of the elements in $S_t$ and $T$.
\item When $k$ is even, and $t$ is even, sample a uniformly random element of $\mathfrak{A}_{[n]\backslash S_t\cup T}$ and view it as an element of $\mathfrak{A}_n$ by fixing all of the elements in $S_t$ and $T$.
\end{enumerate}
\end{defi}
Using natural inclusion from $\mathfrak{S}_{[n]\setminus T}$ to $\mathfrak{S}_n$, we view $\nu_t^T$ as a distribution on $\mathfrak{S}_n$. For the following \Cref{lem: relation of T and M} and \Cref{lem: conditioned}, we choose $ t = t_T$ and use the notation $\nu^T = \nu_{t_T}^T$. For the following lemma, we let $\Omega_{t_T}$ denote the parity class at the relevant actual time $t_T.$
\begin{lemma}\label{lem: relation of T and M}
    For any $\sigma \in \Omega_{t_T}$ such that $\mathrm{Fix}(\sigma) \supseteq [M]$ and $|\mathrm{Fix}(\sigma)| \leq {o((\log n)^{3/2})}$, we have
    \[
    \sum_{[M] \subseteq T \subseteq \mathrm{Fix}(\sigma)}(-1)^{|T| - M}\mathbf{P}[\mathcal{G}(T)]\nu^T(\sigma) = (2+(\log n)^{-1/2+o(1)})\frac{\mathbf{P}[\mathcal{F}(M)]}{(n-M)!}.
    \]
\end{lemma}

\begin{proof}
Notice that $\nu^T(\sigma) = 0$ if $T \not\subseteq \mathrm{Fix}(\sigma)$. Thus, we can only consider $\sigma \in \Omega_{t_T}$ such that $[M], T \subseteq \mathrm{Fix}(\sigma)$ and $|\mathrm{Fix}(\sigma)| \leq o((\log n)^{3/2})$, we see that
    \[
    \nu^T(\sigma) = \sum_{r=0}^{|\mathrm{Fix}(\sigma)| - |T|}\frac{\mathbf{P}(\mathrm{Pois}(\gamma_{n-|T|, t_T}) = r)}{\mathbf{P}(\mathrm{Pois}(\gamma_{n-|T|, t_T}) \leq n-|T|)}\frac{2\binom{|\mathrm{Fix}(\sigma)| -|T|}{r}}{\binom{n-|T|}{r}(n-|T|-r)!}.
    \]
Because $\gamma_{n-|T|, t_T} \leq \sqrt{\log \log n}$ and $|T| \leq o((\log n)^{3/2})$, we have that
\[
\mathbf{P}(\mathrm{Pois}(\gamma_{n-|T|, t_T}) > n-|T|) \leq \left( \frac{e\sqrt{\log \log n}}{n-|T|} \right)^{n-|T|} = O(n^{-1}).
\]
Also,
\[
\frac{\mathbf{P}(\mathrm{Pois}(\gamma_{n-|T|, t_T}) = r)}{n^{|T|-M}\binom{n-|T|}{r}(n-|T|-r)!} = \frac{e^{-\gamma_{n-|T|, t_T}}\gamma_{n-|T|, t_T}^r}{r!(n-|T|-r)!\binom{n-|T|}{r}n^{|T|-M}} = (1+n^{-1+o(1)})\frac{e^{-\gamma_{n-|T|, t_T}}\gamma_{n-|T|, t_T}^r}{(n-M)!}.
\]
Thus,
\[
\nu^T(\sigma) = (1+n^{-1+o(1)})\frac{e^{-\gamma_{n-|T|, t_T}}n^{|T|-M}}{(n-M)!} \sum_{r=0}^{|\mathrm{Fix}(\sigma)| - |T|} 2\gamma_{n-|T|, t_T}^r{\binom{|\mathrm{Fix}(\sigma)| -|T|}{r}}.
\]
Apply the first and second item of Lemma \ref{lemma::Self-reduction lemma} to the following expression, we have
\[
\begin{split}
    &(1+n^{-1+o(1)})(n-M)!\sum_{[M] \subseteq T \subseteq \mathrm{Fix}(\sigma)}(-1)^{|T| - M}\mathbf{P}[\mathcal{G}(T)]\nu^T(\sigma) \\
    &= (1+(\log n)^{-1/2+o(1)})\mathbf{P}[\mathcal{F}(M)]\sum_{[M] \subseteq T \subseteq \mathrm{Fix}(\sigma)}(-\gamma_{n-|T|, t_T})^{|T| - M}\sum_{r=0}^{|\mathrm{Fix}(\sigma)| - |T|} 2 \gamma_{n-|T|, t_T}^{r}\binom{|\mathrm{Fix}(\sigma)| -|T|}{r}\\
    & = (2+(\log n)^{-1/2+o(1)})\mathbf{P}[\mathcal{F}(M)]\sum_{[M] \subseteq T \subseteq \mathrm{Fix}(\sigma)}(-\gamma_{n-|T|, t_T})^{|T| - M}(1+\gamma_{n-|T|, t_T})^{|\mathrm{Fix}(\sigma)| - |T|}\\
    & = (2+(\log n)^{-1/2+o(1)})\mathbf{P}[\mathcal{F}(M)]\sum_{|T| - M \geq 0} \binom{|\mathrm{Fix}(\sigma)| - M}{|T|-M}
    (-\gamma_{n-|T|, t_T})^{|T|-M}(1+\gamma_{n-|T|, t_T})^{|\mathrm{Fix}(\sigma)| - |T|}\\
    &= (2+(\log n)^{-1/2+o(1)})\mathbf{P}[\mathcal{F}(M)].
\end{split}
\]
Here, we used binomial theorem in the second and fourth lines.
\end{proof}

\begin{lemma}\label{lem: conditioned}
Fix $t$ such that $|t'|\leq \frac{n}{k}(\frac{1}{2}\log\log\log n-2\log\log\log\log n)$. For $T \subset [n]$, let $\tilde{X}_t^{T}$ denote the random variable $X_t$ conditioned on the event $\mathcal{G}(T)$. Let $\nu_t^T$ denote the distribution described in \Cref{def: nu^T_t}. If $|T| = n^{o(1)}$, then
\[
d_{\mathrm{TV}}(\tilde{X}_t^T, \nu_t^T) \leq o\left((\log n)^{-5/2}\right).
\]
\end{lemma}

\begin{proof}
    Conditional on the event $\mathcal{G}(T)$, every sampled $k$-cycle is supported on $[n]\setminus T$, and the induced process on $[n]\setminus T$ is exactly the random $k$-cycle walk on $n-|T|$ points, run for $t$ steps.
    On the one hand, recall from \eqref{eq: gamma_T} that our times $t,t_T$ are designed to satisfy
    $$
        \gamma_{n-|T|,t_T}
        =
        \exp\left(
        -\frac{k}{n-|T|}\cdot \frac{n-|T|}{n}t'
        \right)
        =
        \exp\left(-\frac{kt'}{n}\right)
        =
        \gamma_{n,t}.
    $$

    On the other hand, when we apply \Cref{thm: X and nu} to the conditioned walk on
    $[n]\setminus T$ at the actual time $t$, the corresponding time shift is
    $$
        t-\frac{(n-|T|)\log(n-|T|)}{k}.
    $$
    Since
    $$
        t-\frac{(n-|T|)\log(n-|T|)}{k}=t'+O\left(\frac{|T|\log n}{k}+1\right),
    $$
    the corresponding Poisson parameter differs from
    $$
        \exp\left(-\frac{kt'}{n-|T|}\right)
    $$
    by a multiplicative factor $1+n^{-1+o(1)}$. It remains to compare
    $\exp(-kt'/n)$ and $\exp(-kt'/(n-|T|))$. We have
    \begin{align}
        \left|
        \exp\left(-\frac{kt'}{n}\right)
        -
        \exp\left(-\frac{kt'}{n-|T|}\right)
        \right|
        &\leq
        \exp\left(\frac{k|t'|}{n}\right)
        \left|
        1-\exp\left(-\frac{kt'|T|}{n(n-|T|)}\right)
        \right| \notag\\
        &\leq
        2\exp\left(\frac{k|t'|}{n}\right)
        \frac{k|t'||T|}{n(n-|T|)} \notag\\
        &= n^{-1+o(1)}.
    \end{align}
    Therefore, by the standard coupling of Poisson random variables with nearby means, the total variation distance between the two corresponding auxiliary measures is $n^{-1+o(1)}$.

    It remains to check that the actual time $t$ lies in the window where \Cref{thm: X and nu}
    applies to the random $k$-cycle walk on $n-|T|$ points. We need
    \begin{equation}
        \left|
        t-\frac{(n-|T|)\log(n-|T|)}{k}
        \right|
        \leq
        \frac{n-|T|}{2k}\log\log\log (n-|T|).
    \end{equation}
    Since $t=\frac{n\log n}{k}+t'$, it suffices to compare $n\log n$ with $(n-|T|)\log(n-|T|)$. We have
    \begin{equation}
        \left|
        \frac{n\log n-(n-|T|)\log(n-|T|)}{k}
        \right|
        \leq
        O\left(\frac{|T|\log n}{k}\right)
        =
        o\left(\frac{n}{k}\log\log\log\log n\right),
    \end{equation}
    while
    \begin{equation}
        \frac{n-|T|}{2k}\log\log\log(n-|T|)
        =
        \frac{n}{2k}\log\log\log n
        +
        o\left(\frac{n}{k}\log\log\log\log n\right).
    \end{equation}
    Hence the assumed bound on $t'$ implies the desired window condition for $n-|T|$. Applying \Cref{thm: X and nu} to the conditioned walk on $[n]\setminus T$, and then using
    the above comparison of the Poisson parameters, gives
    $$
        d_{\mathrm{TV}}\left(\widetilde X_t^T,\nu_t^T\right)
        \leq
        o\left((\log n)^{-5/2}\right)+n^{-1+o(1)}
        =
        o\left((\log n)^{-5/2}\right).
    $$
    This completes the proof.
\end{proof}

For each integer \(t\), define the parity class
$$
\Omega_t:=
\begin{cases}
\mathfrak A_n, & k \text{ is odd}\\
\mathfrak A_n, & k \text{ is even and } t \text{ is even}\\
\mathfrak A_n^c, & k \text{ is even and } t \text{ is odd}
\end{cases}.
$$
For a subset \(M\subseteq [n]\), set
$$
    \Omega_t(M):=\{\sigma\in \Omega_t:\sigma(i)=i\text{ for every }i\in M\}.
$$
Let \(\mu_t^M\) denote the uniform measure on \(\Omega_t(M)\). Equivalently, \(\mu_t^M\) is obtained by sampling uniformly from the appropriate parity class on \([n]\setminus M\), and then extending the permutation to \([n]\) by fixing every point of \(M\).

In the following proposition, we show that when the time is close to $n\log n/k$, the random walk is nearly uniform after conditioning on the event that the set of untouched cards is exactly $M$.

\begin{prop}\label{prop:key}
    Fix $t$ and $M\subseteq [n]$ such that
    $$
        |t'|\leq \frac{n}{k}
        \left(
        \frac{1}{2}\log\log\log n
        -2\log\log\log\log n
        \right)
        \qquad\text{and}\qquad
        |M|\leq \log n.
    $$
    Let $X_t^M$ denote the random variable $X_t$ conditioned on the event
    $\mathcal F(M)$. Then
    $$
        d_{\mathrm{TV}}(X_t^M,\mu_t^M)
        \leq
        (\log n)^{-1/2+o(1)}.
    $$
\end{prop}

\begin{proof}
    By the definition of total variation distance,
    \begin{align*}
        d_{\mathrm{TV}}(X_t^M,\mu_t^M)
        &=
        \frac{1}{2}
        \sum_{\sigma\in \Omega_t(M)}
        \left|
        \mathbf P(X_t^M=\sigma)-\mu_t^M(\sigma)
        \right| \\
        &=
        \frac{1}{2}
        \sum_{\sigma\in \Omega_t(M)}
        \left|
        \frac{\mathbf P(\{X_t=\sigma\}\cap \mathcal F(M))}
        {\mathbf P(\mathcal F(M))}
        -
        \frac{2}{(n-|M|)!}
        \right|.
    \end{align*}
    Therefore, it suffices to prove
    $$
        \sum_{\sigma\in \Omega_t(M)}
        \left|
        \mathbf P(\{X_t=\sigma\}\cap \mathcal F(M))
        -
        \frac{2\mathbf P(\mathcal F(M))}{(n-|M|)!}
        \right|
        \leq
        (\log n)^{-1/2+o(1)}\mathbf P(\mathcal F(M)).
    $$

    We split $\Omega_t(M)$ according to the number of fixed points. Let
    $$
        S:=
        \Omega_t(M)\cap
        \{\sigma\in \Omega_t:|\operatorname{Fix}(\sigma)|\leq 3\log n\},
        \qquad
        L:=\Omega_t(M)\setminus S.
    $$
    We bound the contributions from $L$ and $S$ separately.

    \medskip
    \noindent
    \textbf{Contribution from $L$.}
    By \Cref{lem: conditioned}, we have
    $$
        \sum_{\sigma\in \Omega_t(M)}
        \left|
        \mathbf P(\widetilde X_t^M=\sigma)-\nu^M(\sigma)
        \right|
        =
        o\left((\log n)^{-5/2}\right).
    $$
    Recall that
    $$
        t_M=\frac{(n-|M|)\log(n-|M|)}{k}
        +\frac{n-|M|}{n}t',
    $$
    and that $\nu^M$ is obtained by first sampling a Poisson number of additional
    fixed points in $[n]\setminus M$, with parameter $\gamma_{n-|M|,t_M}$, and then
    sampling uniformly from the appropriate parity class on the remaining points.
    Hence
    \begin{align*}
        \sum_{\sigma\in L}\mathbf P(\widetilde X_t^M=\sigma)
        &\leq
        \sum_{\sigma\in L}
        \left|
        \mathbf P(\widetilde X_t^M=\sigma)-\nu^M(\sigma)
        \right|
        +\nu^M(L) \\
        &\leq
        o\left((\log n)^{-5/2}\right)
        +
        \mathbf P\left(\operatorname{Pois}(\gamma_{n-|M|,t_M})\geq \log n\right) +
        \sup_{\substack{Q\subseteq [n]\setminus M\\ |Q|\leq \log n}}
        \mu_t^{M\cup Q}(L).\\
    \end{align*}
    Since
    $$
        \gamma_{n-|M|,t_M}
        =
        \gamma_{n,t}
        \leq
        \exp\left(
        \frac{1}{2}\log\log\log n
        -2\log\log\log\log n
        \right)
        \leq
        (\log\log n)^{1/2+o(1)},
    $$
    Chernoff's bound gives
    $$
        \mathbf P\left(\operatorname{Pois}(\gamma_{n-|M|,t_M})\geq \log n\right)
        \leq
        \left(
        \frac{e\sqrt{\log\log n}}{\log n}
        \right)^{\log n}
        \leq
        (\log n)^{-1/2+o(1)}.
    $$
    Moreover, uniformly over all $Q\subseteq [n]\setminus M$ with $|Q|\leq \log n$, a uniform permutation on $[n]\setminus (M\cup Q)$ has at least $\log n$ fixed
    points with probability at most $
    \frac{2}{(\log n)!}\leq(\log n)^{-1/2+o(1)}$. Therefore,
    $$
        \sum_{\sigma\in L}\mathbf P(\widetilde X_t^M=\sigma)
        \leq
        (\log n)^{-1/2+o(1)}.
    $$

    Since $\mathcal F(M)\subseteq \mathcal G(M)$, we have
    \begin{align*}
        \sum_{\sigma\in L}
        \mathbf P(\{X_t=\sigma\}\cap \mathcal F(M))
        &\leq
        \sum_{\sigma\in L}
        \mathbf P(\{X_t=\sigma\}\cap \mathcal G(M)) \\
        &=
        \mathbf P(\mathcal G(M))
        \sum_{\sigma\in L}
        \mathbf P(\widetilde X_t^M=\sigma) \\
        &\leq
        (\log n)^{-1/2+o(1)}\mathbf P(\mathcal G(M)).
    \end{align*}
    By \Cref{lemma::Self-reduction lemma},
    $$
        \mathbf P(\mathcal G(M))
        \leq
        (\log n)^{o(1)}\mathbf P(\mathcal F(M)).
    $$
    Hence
    $$
        \sum_{\sigma\in L}
        \mathbf P(\{X_t=\sigma\}\cap \mathcal F(M))
        \leq
        (\log n)^{-1/2+o(1)}\mathbf P(\mathcal F(M)).
    $$
    On the other hand, since $\mu_t^M$ is uniform on $\Omega_t(M)$,
    $$
        \sum_{\sigma\in L}
        \frac{2\mathbf P(\mathcal F(M))}{(n-|M|)!}
        =
        \mathbf P(\mathcal F(M))\mu_t^M(L)
        \leq
        (\log n)^{-1/2+o(1)}\mathbf P(\mathcal F(M)).
    $$
    Combining the preceding two estimates gives
    $$
        \sum_{\sigma\in L}
        \left|
        \mathbf P(\{X_t=\sigma\}\cap \mathcal F(M))
        -
        \frac{2\mathbf P(\mathcal F(M))}{(n-|M|)!}
        \right|
        \leq
        (\log n)^{-1/2+o(1)}\mathbf P(\mathcal F(M)).
    $$

    \medskip
    \noindent
    \textbf{Contribution from $S$.}
    For $\sigma\in S$, the inclusion-exclusion identity gives
    $$
        \mathbf 1_{\mathcal F(M)}
        =
        \sum_{T\supseteq M}
        (-1)^{|T|-|M|}\mathbf 1_{\mathcal G(T)}.
    $$
    Since $\nu^T(\sigma)=0$ unless $T\subseteq \operatorname{Fix}(\sigma)$, we may restrict
    the sum to $M\subseteq T\subseteq \operatorname{Fix}(\sigma)$. By \Cref{lem: relation of T and M},
    for every $\sigma\in S$ we have
    $$
        \sum_{M\subseteq T\subseteq \operatorname{Fix}(\sigma)}
        (-1)^{|T|-|M|}
        \mathbf P(\mathcal G(T))\nu^T(\sigma)
        =
        \left(2+(\log n)^{-1/2+o(1)}\right)
        \frac{\mathbf P(\mathcal F(M))}{(n-|M|)!}.
    $$
    Therefore,
    \begin{align*}
        &\sum_{\sigma\in S}
        \left|
        \mathbf P(\{X_t=\sigma\}\cap \mathcal F(M))
        -
        \frac{2\mathbf P(\mathcal F(M))}{(n-|M|)!}
        \right| \\
        &\leq
        (\log n)^{-1/2+o(1)}\mathbf P(\mathcal F(M))+
        \sum_{\sigma\in S}
        \sum_{M\subseteq T\subseteq \operatorname{Fix}(\sigma)}
        \left|
        \mathbf P(\{X_t=\sigma\}\cap \mathcal G(T))
        -
        \mathbf P(\mathcal G(T))\nu^T(\sigma)
        \right| \\
        &\leq
        (\log n)^{-1/2+o(1)}\mathbf P(\mathcal F(M))\quad+
        \sum_{\substack{T\supseteq M\\ |T|\leq 3\log n}}
        \mathbf P(\mathcal G(T))
        \sum_{\sigma\in S}
        \left|
        \mathbf P(X_t=\sigma\mid \mathcal G(T))
        -
        \nu^T(\sigma)
        \right|. \\
    \end{align*}
    By \Cref{lem: conditioned}, for every $T$ appearing in the last sum,
    $$
        \sum_{\sigma\in S}
        \left|
        \mathbf P(X_t=\sigma\mid \mathcal G(T))
        -
        \nu^T(\sigma)
        \right|
        \leq
        2d_{\mathrm{TV}}(\widetilde X_t^T,\nu^T)
        =
        o\left((\log n)^{-5/2}\right).
    $$
    Hence, using the last item of \Cref{lemma::Self-reduction lemma},
    \begin{align*}
        &\sum_{\sigma\in S}
        \left|
        \mathbf P(\{X_t=\sigma\}\cap \mathcal F(M))
        -
        \frac{2\mathbf P(\mathcal F(M))}{(n-|M|)!}
        \right| \\
        &\leq
        (\log n)^{-1/2+o(1)}\mathbf P(\mathcal F(M))
        +
        o\left((\log n)^{-5/2}\right)
        \sum_{\substack{T\supseteq M\\ |T|\leq 3\log n}}
        \mathbf P(\mathcal G(T)) \\
        &\leq
        (\log n)^{-1/2+o(1)}\mathbf P(\mathcal F(M)).
    \end{align*}

    Combining the contributions from $S$ and $L$ proves the desired estimate, and hence
    completes the proof.
\end{proof}

\subsection{The marking schemes for odd and even $k$-cycles}\label{subsec: marking schemes}

In this subsection, we introduce the marking schemes used in the proof of \Cref{thm: hitting time mixing_main thm}. The role of these schemes is to turn the conditional uniformity statement from \Cref{prop:key} into a stopped-time statement. We start the marking procedure at the deterministic time
$t^*$ defined in \eqref{eq: starting time of marking scheme}. At this time, all cards which have already been touched are declared marked, and the remaining cards are marked one by one during the subsequent evolution.

The construction has to serve two purposes at once. On the one hand, it should remain close to the genuine touching process of the random $k$-cycle walk, so that the terminal marking time is close to the hitting time $\tau$. On the other
hand, it should be artificial enough to preserve an exact conditional uniformity property after conditioning on the current marked set. We therefore use two related marking schemes. Marking scheme B uses the original transition kernel $P_n$ and is coupled directly to the original walk. Marking scheme A uses a
modified time-inhomogeneous transition kernel $Q_{n,t}$, designed so that an auxiliary process remains conditionally uniform on the appropriate parity class.

More precisely, the proof will compare four processes
$$
\rho_t\longleftrightarrow Y_t\longleftrightarrow Z_t\longleftrightarrow X_t .
$$
Here $X_t$ is the original random $k$-cycle walk. The process $Z_t$ uses marking
scheme B and is therefore closest to $X_t$. The process $Y_t$ uses marking scheme
A and is coupled to $Z_t$ by comparing the kernels $Q_{n,t}$ and $P_n$.
Finally, $\rho_t$ uses the same $Q_{n,t}$-moves and the same marked sets as
$Y_t$, but its action on the permutation is modified in order to maintain exact
conditional uniformity. Thus $\rho_t$ is the process which supplies the uniform
distribution at the terminal marking time, while $Y_t$ and $Z_t$ bridge it back
to the original chain.

For the rest of this subsection, let $S_{t^*}$ denote the set of elements in
$[n]$ which have not been touched by time $t^*$. For $t\ge t^*$, let $S_t$
denote the set of unmarked elements and let
$$
M_t=[n]\setminus S_t
$$
be the marked set. The odd and even cases have the same basic marking structure:
a useful update is one which contains exactly one element of $S_t$, in which
case that element is marked. The only additional issue is parity. When $k$ is
odd, each $k$-cycle is an even permutation, so the walk stays in a fixed parity
class. When $k$ is even, each $k$-cycle is odd, so the auxiliary chain must also
change parity at every step. This is the reason for the parity-correcting
transpositions supported on the marked set in the definition of $Q_{n,t}$ below.

\begin{defi}[The modified transition kernel $Q_{n,t}$]
For $t\ge t^*$, set
$$
C_t=(|M_t|-1)_{k-2},
\qquad
q_t=\frac{|M_t|}{|M_t|+1}
\left(\frac{n!}{k(n-k)!}\right)^{-1}.
$$
If $k$ is odd, then the kernel $Q_{n,t}$ is defined by
\[
Q_{n,t}=\begin{cases}
(i\cdots i) & \text{with probability } C_tq_t \ \text{for } i\in S_t,\\[6pt]
(i_1\cdots i_k) & \text{with probability } q_t\ \text{for } i_1\in S_t,\ i_2,\ldots,i_k\in M_t,\\[6pt]
(i_1\cdots i_k) & \text{with probability } \left(\dfrac{n!}{k(n-k)!}\right)^{-1}\ \text{otherwise}.
\end{cases}
\]
If $k$ is even, then the kernel $Q_{n,t}$ is defined by
\[
Q_{n,t}=\begin{cases}
(ab)(i\cdots i) & \text{with probability } \dfrac{C_tq_t}{\binom{|M_t|}{2}}
\ \text{for } i\in S_t,\ \{a,b\}\subset M_t,\\[8pt]
(i_1\cdots i_k) & \text{with probability } q_t\ \text{for } i_1\in S_t,\ i_2,\ldots,i_k\in M_t,\\[6pt]
(i_1\cdots i_k) & \text{with probability } \left(\dfrac{n!}{k(n-k)!}\right)^{-1}\ \text{otherwise}.
\end{cases}
\]

Here $(i\cdots i)$ is a labelled dummy move. In the odd case it acts as the identity. In the even case $(ab)(i\cdots i)$ acts on the permutation as the transposition $(ab)$ and carries the label $i$, which indicates the element to be marked. Throughout the marking argument we restrict to the event $|S_{t^*}|\leq \log n$, whose complement will be handled separately. Under both marking schemes the number of unmarked elements never increases, so for all sufficiently large $n$ we have $|M_t|\geq n-\log n$ for every $t\geq t^*$. Hence all quantities appearing in the definition of $Q_{n,t}$ are well-defined.

Let us check that $Q_{n,t}$ is indeed a probability kernel. The total mass assigned by the original uniform $k$-cycle kernel to cycles with exactly one element in $S_t$ is
$$
|S_t|\cdot|M_t|\cdot(|M_t|-1)_{k-2}
\left(\frac{n!}{k(n-k)!}\right)^{-1}=|S_t|\cdot|M_t|\cdot C_t\cdot
\left(\frac{n!}{k(n-k)!}\right)^{-1}.
$$
Under $Q_{n,t}$, this mass is replaced by two contributions. The ordinary one-unmarked $k$-cycles have total mass $|S_t|\cdot|M_t|\cdot C_t q_t$,
while the dummy moves have total mass
$|S_t|\cdot C_t q_t$ in the odd case, and
$$
|S_t|\cdot \binom{|M_t|}{2}\cdot\frac{C_tq_t}{\binom{|M_t|}{2}}=|S_t|\cdot C_t q_t
$$
in the even case. Hence the total replacement mass is
$$
|S_t|\cdot (|M_t|+1)\cdot C_tq_t=|S_t|\cdot|M_t|\cdot C_t
\left(\frac{n!}{k(n-k)!}\right)^{-1},
$$
by the definition of $q_t$. All other $k$-cycles retain their original mass. Therefore the total mass of $Q_{n,t}$ is indeed one.

\end{defi}

\begin{defi}[Marking scheme A]
At time $t^*$, mark all elements in $[n]\setminus S_{t^*}$. At time $t+1$,
sample a move $\sigma$ from the kernel $Q_{n,t}$.

The marked set is updated as follows.
\begin{itemize}
\item If the sampled move is a dummy move labelled by $i\in S_t$, then mark
$i$.
\item If $\sigma=(i_1\cdots i_k)$ with $i_1\in S_t$ and
$i_2,\ldots,i_k\in M_t$, then mark $i_1$. Thus, $M_{t+1} = M_t \cup \{i_1\}$ and $S_{t+1} = S_t \setminus \{i_1\}$.
\item In all remaining cases, namely when the ordinary $k$-cycle does not
mark exactly one new element, pull back the marked set under $\sigma$:
$$
M_{t+1}=\{i\in[n]:\sigma(i)\in M_t\}=\sigma^{-1}(M_t),
\qquad
S_{t+1}=\{i\in[n]:\sigma(i)\in S_t\}=\sigma^{-1}(S_t).
$$
\end{itemize}
\end{defi}

\begin{defi}[Marking scheme B]
At time $t^*$, mark all elements in $[n]\setminus S_{t^*}$. At time $t+1$,
sample a genuine $k$-cycle
$$
\sigma=(i_1\cdots i_k)
$$
from the original transition kernel $P_n$.

The marked set is updated by the same rule as in marking scheme A, except that
there are no dummy moves. Thus, if exactly one point of the cycle is in $S_t$,
say $i_1\in S_t$ and $i_2,\ldots,i_k\in M_t$, then mark $i_1$ and let $M_{t+1} = M_t \cup \{i_1\}$ and $S_{t+1} = S_t \setminus \{i_1\}$. Otherwise, pull
back the marked set:
$$
M_{t+1}=\sigma^{-1}(M_t),
\qquad
S_{t+1}=\sigma^{-1}(S_t).
$$
\end{defi}

We now define the coupled processes which will be compared in the proof. Let
$$
\rho_{t^*}=Y_{t^*}=Z_{t^*}
$$
be a common random variable sampled uniformly from
$\Omega_{t^*}(S_{t^*})$.

\begin{defi}[The coupled processes]
We define four Markov processes
$$
(X_t)_{t\ge t^*},\qquad
(Y_t)_{t\ge t^*},\qquad
(Z_t)_{t\ge t^*},\qquad
(\rho_t)_{t\ge t^*}
$$
on a common probability space as follows.
\begin{itemize}
\item The process $(X_t)_{t\ge t^*}$ is the original random $k$-cycle walk,
evolving according to the transition kernel $P_n$. Let $\tau$ denote
the first time at which all elements have been touched.

\item The process $(Z_t)_{t\ge t^*}$ is initialized by $Z_{t^*}$ and evolves
according to the original transition kernel $P_n$. Its marked set is
updated according to marking scheme B. Let $\tau_1$ denote the first time at
which all elements have been marked.

\item The process $(Y_t)_{t\ge t^*}$ is initialized by $Y_{t^*}$ and evolves
according to the modified transition kernel $Q_{n,t}$. Its marked set is
updated according to marking scheme A. Let $\tau_2$ denote the first time at
which all elements have been marked.

\item The process $(\rho_t)_{t\ge t^*}$ is initialized by $\rho_{t^*}$ and
uses the same $Q_{n,t}$-moves and the same marked sets as $Y_t$. Its action
on the current state is modified as follows.

At time $t+1$, sample $\sigma\sim Q_{n,t}$.

\begin{itemize}
    \item If the sampled move marks exactly one new element $i$, either through
    a dummy move or through a one-unmarked $k$-cycle, set
    $$
      M_{t+1}:=M_t \cup \{i\},
        \qquad
        S_{t+1}:=S_t \setminus \{i\}, \qquad  \rho_{t+1}:=\sigma\rho_t.
    $$

    \item If no new element is marked and
    $\operatorname{supp}(\sigma)\subset M_t$, set
    $$
       M_{t+1}:=M_t,
        \qquad
        S_{t+1}:=S_t, \qquad  \rho_{t+1}:=\sigma\rho_t.
    $$

    \item If no new element is marked and
    $\operatorname{supp}(\sigma)\not\subset M_t$, first update
    $$
        M_{t+1}:=\sigma^{-1}(M_t),
        \qquad
        S_{t+1}:=\sigma^{-1}(S_t).
    $$
    If $k$ is odd, set
    $$
        \rho_{t+1}:=\sigma^{-1}\rho_t\sigma.
    $$
    If $k$ is even, then, conditionally on $M_{t+1}$, choose an unordered pair
    $\{U_{t+1},V_{t+1}\}\subset M_{t+1}$ uniformly and independently of
    $\rho_t$, and set
    $$
        \rho_{t+1}:=(U_{t+1}\,V_{t+1})\sigma^{-1}\rho_t\sigma.
    $$
\end{itemize}
Let $\tau_3$ denote the first time at which all elements have been marked
for the process $\rho_t$.

\end{itemize}
\end{defi}

\subsection{Proof of \Cref{thm: hitting time mixing_main thm}}
Finally, we are ready for the proof of \Cref{thm: hitting time mixing_main thm}. First, we show that, conditioned on the marked set \(M_t\), the random permutation \(\rho_t\) is uniform on \(\Omega_t(S_t)\). In particular, at the marking time \(\tau_3\), all elements are marked, and hence \(\rho_{\tau_3}\sim U_{\mathfrak A_n}\) when $k$ is odd, and   \(\rho_{\tau_3}\sim \mathbf{P}(\tau_3 ~\text{even})\,U_{\mathfrak A_n}+\mathbf{P}(\tau_3 ~\text{odd})\,U_{\mathfrak A_n^c}\) when $k$ is even.

Second, we compare the auxiliary chains with the original chain. More precisely, we show that
\[
    d_{\mathrm{TV}}(\rho_{\tau_3},Y_{\tau_2}),\qquad
    d_{\mathrm{TV}}(Y_{\tau_2},Z_{\tau_1}),\qquad
    d_{\mathrm{TV}}(Z_{\tau_1},X_{\tau})
\]
are all small. Combining these estimates with the uniformity of \(\rho_{\tau_3}\), together with the fact that in the even-\(k\) case the hitting time has asymptotically balanced parity, will yield the desired hitting-time mixing statement. The start coupling time $t^*$ will again be the same as in \eqref{eq: starting time of marking scheme}. The following lemma is the key reason for introducing the modified process $\rho_t$: it shows that the artificial dynamics preserve exact conditional uniformity, so that once all elements have been marked, the terminal distribution is the desired uniform measure on the appropriate parity class.

\begin{lemma}\label{lem:uniform}
    For $t \geq t^*$, conditioned on the marked set $M_t$, $\rho_t$ is uniform on $\Omega_{t}(S_t)$. In particular, $\rho_{\tau_3}|\{\tau_3= t\}$ is uniform on $\Omega_t$. Consequently, if $k$ is odd, then
    \begin{align*}
        \rho_{\tau_3}\sim U_{\mathfrak{A}_n}.
    \end{align*}
    If $k$ is even, then
\begin{align*}
    \rho_{\tau_3}\sim \mathbf{P}(\tau_3\text{ is even}) \, U_{\mathfrak{A}_n}+\mathbf{P}(\tau_3\text{ is odd}) \, U_{\mathfrak{A}_n^c}.
\end{align*}
\end{lemma}
\begin{proof}
    At time $t^*$, the claim follows from the construction of $\rho_{t^*}$. Suppose that the claim holds at time $t\geq t^*$. We prove it at time $t+1$.

First consider the case in which no new element is marked. There are two subcases. If the sampled $k$-cycle satisfies $\operatorname{supp}(\sigma)\subset M_t$, then $M_{t+1} = M_t, S_{t+1}= S_t,\rho_{t+1}= \sigma\rho_t$. Left multiplication by $\sigma$ gives a bijection $\Omega_t(S_t)\longrightarrow \Omega_{t+1}(S_{t+1})$. Hence $\rho_{t+1}$ is uniform on $\Omega_{t+1}(S_{t+1})$. On the other hand, if the sampled $k$-cycle satisfies $\operatorname{supp}(\sigma)\not\subset M_t$, then $M_{t+1} = \sigma^{-1}(M_t), S_{t+1}= \sigma^{-1}(S_t)$. If $k$ is odd, then $\rho_{t+1}= \sigma^{-1}\rho_t\sigma$. For every sampled move $\sigma$, conjugation by $\sigma$ gives a bijection $\Omega_{t}(S_t)\rightarrow \Omega_{t+1}(S_{t+1})$. Therefore, $\rho_{t+1}$ is uniform on $\Omega_{t+1}(S_{t+1})$. If $k$ is even, then the random permutation
$
\sigma^{-1}\rho_t\sigma
$
is uniform on
$
\sigma^{-1}\Omega_t(S_t)\sigma=\Omega_t(S_{t+1}).
$ Conditionally on $M_{t+1}$, the unordered pair
$
\{U_{t+1},V_{t+1}\}
$
is chosen uniformly from the set of unordered pairs contained in $M_{t+1}$, independently of
$\rho_t$, and
\[
\rho_{t+1}=(U_{t+1}V_{t+1})\sigma^{-1}\rho_t\sigma.
\]
For every fixed unordered pair $\{u,v\}\subset M_{t+1}$, left multiplication by the transposition
$(uv)$ gives a bijection
$
\Omega_t(S_{t+1})\longrightarrow \Omega_{t+1}(S_{t+1}),
$
because $(uv)$ is supported on $M_{t+1}$ and changes parity. Averaging over the uniformly
chosen pair therefore preserves the uniform law on $\Omega_{t+1}(S_{t+1})$.

It remains to consider the case in which exactly one new element is marked. Then there exists
$i\in S_t$ such that
\[
M_{t+1}=M_t\cup\{i\},\qquad S_{t+1}=S_t\setminus\{i\}.
\]
Fix $\pi\in \Omega_{t+1}(S_{t+1})$. Since $\pi$ fixes every element of
$S_t\setminus\{i\}$, bijectivity implies $\pi(i)\in M_t\cup\{i\}$. 
We show
that the total transition weight of all moves taking a state in
$\Omega_t(S_t)$ to $\pi$ is independent of $\pi$.

Let $\alpha$ denote the sampled move acting on $\rho_t$, and write
$m=|M_t|$. If $k$ is odd, the moves that mark $i$ are the identity dummy move,
with weight $C_tq_t$, and the one-unmarked cycles
$(i\,a_2\,\cdots\,a_k)$ with $a_2,\ldots,a_k\in M_t$ distinct, each with
weight $q_t$. If $\pi(i)=i$, only the identity dummy move contributes, so the
total contribution is $C_tq_t=(m-1)_{k-2}q_t$. If $\pi(i)\in M_t$, then a
contributing genuine cycle must satisfy $\alpha(i)=\pi(i)$; hence
\[
    \alpha=(i\,\pi(i)\,a_3\,\cdots\,a_k),
\]
where $a_3,\ldots,a_k$ are distinct elements of
$M_t\setminus\{\pi(i)\}$. There are $(m-1)_{k-2}=C_t$ such cycles, so the total
contribution is again $C_tq_t$.

If $k$ is even, the moves that mark $i$ are the dummy transpositions
$(ab)(i\cdots i)$, where $\{a,b\}\subset M_t$, each with weight
$C_tq_t/\binom{m}{2}$, and the same one-unmarked genuine cycles as above, each
with weight $q_t$. If $\pi(i)=i$, only dummy transpositions can contribute, and
their total contribution is
\[
    \binom{m}{2}\cdot \frac{C_tq_t}{\binom{m}{2}}=C_tq_t.
\]
If $\pi(i)\in M_t$, no dummy transposition can contribute because every dummy
transposition fixes $i$. The same counting of genuine cycles as in the odd case
therefore gives total contribution $C_tq_t$.

Thus, in both parity cases, the total preimage weight is independent of the
target permutation $\pi\in\Omega_{t+1}(S_{t+1})$. Consequently, conditionally on
marking the element $i$, the random permutation $\rho_{t+1}$ is uniform on
$\Omega_{t+1}(S_{t+1})$.
\end{proof} 


The next lemma compares marking scheme A with marking scheme B by showing that the corresponding terminal times and stopped permutations agree with high probability under a natural coupling.

\begin{lemma}\label{lem: Z and Y}
    With notations defined in \Cref{subsec: marking schemes}, we have $\mathbf P(\tau_1\neq \tau_2)\le n^{-1+o(1)}$
    and
    $$d_{\mathrm{TV}}(\mathcal{L}(Z_{\tau_1}),\mathcal{L}(Y_{\tau_2})) \leq n^{-1+o(1)}.$$
\end{lemma}

\begin{proof}
    Let $\mathcal E:=\{|S_{t^*}|>\log n\}$.
    For $m\geq 0$, let $B_m$ be the event that, in the natural step-by-step coupling of the $Z$-chain and the $Y$-chain, the two chains use different updates at least once during the time interval from $t^*$ to $t^*+m$.

    We first record the consequence of the coupling on the good event
    $$
    \mathcal E^c\cap B^c_{\frac{n\log n}{k}}\cap
    \left\{
    \max\{\tau_1,\tau_2\}-t^*\leq \frac{n\log n}{k}
    \right\}.
    $$
    On this event, the number of initially unmarked cards is at most $\log n$,
    and the two chains use the same update at every step until both stopping
    times have occurred. Since marking schemes A and B mark exactly the same
    card whenever the same admissible update is used, the marked sets in the
    two chains evolve identically throughout this time interval. Therefore the
    two terminal marking times are equal, namely
    $$
    \tau_1=\tau_2,
    $$
    and the corresponding permutations also agree:
    $$
    Z_{\tau_1}=Y_{\tau_2}.
    $$
    Hence
    $$
    \{\tau_1\neq \tau_2\}\cup\{Z_{\tau_1}\neq Y_{\tau_2}\}
    \subset
    \mathcal E\cup
    \left(B_{\frac{n\log n}{k}}\cap \mathcal E^c\right)
    \cup
    \left\{
    \max\{\tau_1,\tau_2\}-t^*>\frac{n\log n}{k}
    \right\}.
    $$
    It follows that
    $$
    \mathbf P(\tau_1\neq \tau_2)
    \leq
    \mathbf P(\mathcal E)
    +
    \mathbf P\left(B_{\frac{n\log n}{k}}\cap \mathcal E^c\right)
    +
    \mathbf P\left(
    \max\{\tau_1,\tau_2\}-t^*>\frac{n\log n}{k}
    \right).
    $$
    The same bound also controls the coupling error between the stopped
    variables. By the coupling inequality,
    $$
    d_{\mathrm{TV}}\bigl(\mathcal L(Z_{\tau_1}),\mathcal L(Y_{\tau_2})\bigr)
    \leq
    \mathbf P(Z_{\tau_1}\neq Y_{\tau_2})
    $$
    and hence
    $$
    d_{\mathrm{TV}}\bigl(\mathcal L(Z_{\tau_1}),\mathcal L(Y_{\tau_2})\bigr)
    \leq
    \mathbf P(\mathcal E)
    +
    \mathbf P\left(B_{\frac{n\log n}{k}}\cap \mathcal E^c\right)
    +
    \mathbf P\left(
    \max\{\tau_1,\tau_2\}-t^*>\frac{n\log n}{k}
    \right).
    $$

    We now bound the three terms on the right-hand side. First, by the estimate
    on the number of cards not yet touched at time $t^*$, we have
    \begin{equation}\label{eq: St small}
        \mathbf P(\mathcal E)=\mathbf P(|S_{t^*}|>\log n)=n^{-\omega(1)}.
    \end{equation}
    Second, on the event $\mathcal E^c$, the number of unmarked cards is at most
    $\log n$ throughout the coupling. For simpler notation, we let \(s = |S_t|\) and \(|M_t| = n-s\). 
Then we compare the one-step transition rules of the two coupled chains and get
\begin{equation}
\begin{aligned}
d_{\mathrm{TV}}\!\left(P_{n,t},Q_{n,t}\right)= s C_t q_t 
= \frac{s}{n(n-1)}
   \frac{(n-s)k}{n-s+1}
   \prod_{i=2}^{k-1}
   \left(1-\frac{s-1}{n-i}\right) = \frac{sk}{n^2}\left(1+o(1)\right),
\end{aligned}
\tag{3.16}
\end{equation}
for \(k=o\!\left(n/(\log n)^4\right)\) and \(s=O(\log n)\).
    
We obtain by a union bound over the next
    $n\log n/k$ steps that
    $$
    \mathbf P\left(B_{\frac{n\log n}{k}}\cap \mathcal E^c\right)
    \leq
    \frac{n\log n}{k}\cdot
    \frac{k\log n}{n^2}(1+o(1))
    =
    n^{-1+o(1)}.
    $$
    Finally, we use the coupon-collector-type estimates for the two marking schemes. Fix a point in $S_{t^*}$. At each
subsequent step, this point is included in the chosen uniform $k$-cycle with probability $k/n$.
Hence, after the next $\left\lceil n\log n/k\right\rceil$ steps, the probability that this point
is still untouched is at most
\begin{equation}
    \left(1-\frac{k}{n}\right)^{\left\lceil n\log n/k\right\rceil}
    \leq \exp(-\log n)
    = n^{-1}.
\end{equation}
Taking a union bound over all points in $S_{t^*}$ gives
\begin{equation}\label{eq: coupon}
    \mathbf P\left(\tau-t^*>\frac{n\log n}{k},\ |S_{t^*}|\leq \log n\right)\leq\log n\cdot\left(1-\frac{k}{n}\right)^{\left\lceil n\log n/k\right\rceil}\leq\frac{\log n}{n}
    = n^{-1+o(1)}.
\end{equation}
As a result, we get
    $$
    \mathbf P\left(
    \tau_1-t^*>\frac{n\log n}{k},\,\mathcal E^c
    \right)
    \leq n^{-1+o(1)}\quad \text{and} \quad
    \mathbf P\left(
    \tau_2-t^*>\frac{n\log n}{k},\,\mathcal E^c
    \right)
    \leq n^{-1+o(1)}.
    $$
    Therefore
    $$
    \mathbf P\left(
    \max\{\tau_1,\tau_2\}-t^*>\frac{n\log n}{k}
    \right)
    \leq
    n^{-1+o(1)}+\mathbf P(\mathcal E)
    =
    n^{-1+o(1)}.
    $$

    Combining the preceding estimates yields
    $$
    \mathbf P(\tau_1\neq \tau_2)\leq n^{-1+o(1)}
    $$
    and
    $$
    d_{\mathrm{TV}}\bigl(\mathcal L(Z_{\tau_1}),\mathcal L(Y_{\tau_2})\bigr)
    \leq n^{-1+o(1)}.
    $$
    This completes the proof.
\end{proof}


We next compare the terminal time of marking scheme B with the genuine hitting time $\tau$ of the original random $k$-cycle walk.

\begin{lemma}\label{lemma: Z,X gap}
    With notations defined in \Cref{subsec: marking schemes}, we have 
    $$d_{\mathrm{TV}}(\mathcal{L}(Z_{\tau_1}),\mathcal{L}(X_{\tau}))\leq \exp\left(-(\log\log n)^{1/2+o(1)}\right).$$
\end{lemma}

\begin{proof}
Let $\mathcal B$ be the event that, before time
$t^*+ n\log n/k$, some chosen $k$-cycle contains at least two elements of $S_{t^*}$ conditioned on $|S_{t^*}|\leq \log n$. Then, by a union bound,
$$
\mathbf P(\tau_1\neq \tau)\leq\mathbf P(\tau\leq t^*)+\mathbf P(|S_{t^*}|\geq \log n) +\mathbf P\left(\tau-t^*>\frac{n\log n}{k},\ |S_{t^*}|\leq \log n\right)+\mathbf P(\mathcal B).
$$
By \Cref{lemma::Self-reduction lemma}, we have
$$
    \mathbf P(\tau\leq t^*)
    \leq \mathbf P(\mathcal F^{t^*}(\varnothing)) =\left(1+(\log n)^{-1/2+o(1)}\right)
    \exp(-\gamma_{t^*}).
$$
Moreover,
\begin{align*}
    \gamma_{t^*}
    &=\exp\left(
    -\frac{\left(t^*-\frac{n\log n}{k}\right)k}{n}
    \right) \\
    &=\exp\left(
    \frac{1}{2}\log\log\log n
    -4\log\log\log\log n
    +o(1)
    \right) \\
    &=(\log\log n)^{1/2+o(1)}.
\end{align*}
Therefore,
\begin{equation}\label{eq: finish early}
    \mathbf P(\tau\leq t^*)
    \leq \exp\left(-(\log\log n)^{1/2+o(1)}\right).
\end{equation}
Also, by the definition of $t^*$ and the coupon-collector estimate for the number of untouched
points,
\begin{equation}
    \mathbf P(|S_{t^*}|\geq \log n)\leq n^{-\omega(1)}.
\end{equation}
We now condition on the event $|S_{t^*}|\leq \log n$.
By \eqref{eq: coupon}, we see that
$$\mathbf P\left(\tau-t^*>\frac{n\log n}{k},\ |S_{t^*}|\leq \log n\right)\leq\frac{\log n}{n}
    = n^{-1+o(1)}.$$

It remains to control the event that the touching scheme and the marking scheme differ before
time $t^*+ n\log n/k$. Conditional on $|S_{t^*}|\leq \log n$, at a single
step this can happen only if the chosen uniform $k$-cycle contains at least two elements of
$S_{t^*}$. The probability of this event is at most
\begin{equation}
    \binom{k}{2}\left(\frac{|S_{t^*}|}{n}\right)^2
    \leq \binom{k}{2}\frac{(\log n)^2}{n^2}.
\end{equation}
Summing over the next $\left\lceil n\log n/k\right\rceil$ steps, we obtain
\begin{equation}\label{eq: prob of B}
\mathbf P(\mathcal B)\leq\left\lceil\frac{n\log n}{k}\right\rceil\binom{k}{2}\frac{(\log n)^2}{n^2} \notag\leq \frac{k(\log n)^3}{n}\leq \frac{1}{\log n}.
\end{equation}
Combining the preceding estimates yields
\begin{equation}
    \mathbf P(\tau_1\neq \tau)
    \leq \exp\left(-(\log\log n)^{1/2+o(1)}\right).
\end{equation}

We first compare the initial laws of \(Z_{t^*}\) and \(X_{t^*}\). Recall that \(S_{t^*}\) denotes the set of cards which have not been touched by time \(t^*\). Conditional on the event \(S_{t^*}=M\), the random variable \(Z_{t^*}\) is sampled uniformly from \(\Omega_{t^*}(M)\), and hence has law \(\mu_{t^*}^M\). On the other hand, conditional on \(S_{t^*}=M\), the law of \(X_{t^*}\) is exactly the law of \(X_{t^*}\) conditioned on \(F(M)\). Therefore, by conditioning on \(S_{t^*}\), we obtain
$$
\begin{aligned}
    d_{\mathrm{TV}}\bigl(\mathcal L(Z_{t^*}),\mathcal L(X_{t^*})\bigr)
    &\leq
    \sum_{M\subseteq[n]}
    \mathbf P(S_{t^*}=M)\,
    d_{\mathrm{TV}}\bigl(\mu_{t^*}^M,\mathcal L(X_{t^*}\mid F(M))\bigr) \\
    &\leq
    (\log n)^{-1/2+o(1)}
    +
    \mathbf P(|S_{t^*}|>\log n).
\end{aligned}
$$
Here, in the second line we used \Cref{prop:key} for all \(M\) with \(|M|\leq \log n\), and the trivial bound \(d_{\mathrm{TV}}\leq 1\) for the remaining values of \(M\). By the coupon-collector estimate at time \(t^*\), we have
$$
\mathbf P(|S_{t^*}|>\log n)=n^{-\omega(1)}.
$$
Thus
$$
d_{\mathrm{TV}}\bigl(\mathcal L(Z_{t^*}),\mathcal L(X_{t^*})\bigr)
\leq
(\log n)^{-1/2+o(1)}+n^{-\omega(1)}.
$$

We now compare the laws at the stopping times. Consider a coupling of $(Z_t,X_t)$ such that $d_{TV}(\mathcal{L}(Z_{t^*}),\mathcal{L}(X_{t^*})) = \mathbf{P}(Z_{t^*} \neq X_{t^*})$ and $Z_t$, $X_t$ evolve according to the same transition kernel after time \(t^*\). The total variation distance between their laws cannot increase under this common Markov evolution. Then

$$
\begin{aligned}
    &d_{\mathrm{TV}}\bigl(\mathcal L(Z_{\tau_1}),\mathcal L(X_\tau)\bigr)
    \leq
    d_{\mathrm{TV}}\bigl(\mathcal L(Z_{t^*}),\mathcal L(X_{t^*})\bigr)
    +\mathbf P(\tau_1\neq \tau)+  \mathbf P(\max(\tau_1,\tau) - t^*> n\log n/k) \\
    &\leq
    (\log n)^{-1/2+o(1)}
    +n^{-\omega(1)}
    +\exp\left(-(\log\log n)^{1/2+o(1)}\right)  + n^{-1+o(1)}\\
    &\leq
    \exp\left(-(\log\log n)^{1/2+o(1)}\right).
\end{aligned}
$$
\end{proof}

The next lemma shows that the auxiliary uniform process $\rho_t$ and the $Y$-chain have the same terminal marking time, and remain close at that time.

\begin{lemma}\label{lem: Y and rho}
    Under the natural coupling in which $(Y_t)_{t\ge t^*}$ and $(\rho_t)_{t\ge t^*}$ both use the same sampled move generated from $Q_{n,t}$ at each time and $Y_{t^*} = \rho_{t^*}$, one has $\tau_2=\tau_3$. Moreover,
    \[
    d_{\mathrm{TV}}(\mathcal L(Y_{\tau_2}), \mathcal L(\rho_{\tau_3})) \leq \frac{1+o(1)}{\log n}.
    \]
\end{lemma}

\begin{proof}
    Because we have coupled the two chains using the same sampled move $\sigma_t \sim Q_{n,t}$ at each time $t \geq t^*$ and set $Y_{t^*} = \rho_{t^*}$. The marked sets for these two chains are identical at any time and therefore $\tau_2 = \tau_3$.   
    Let $T = \frac{n \log n}{k} + t^*$. We know that
    \[
    \mathbf{P}\left(\tau_2  \geq T \right) \leq n^{-1+o(1)},\quad \mathbf{P}(|S_{t^*}| \geq \log n) \leq n ^{- \omega(1)}.
    \]
    Let $\mathcal{B}_T$ be the event that there exists some $1 \leq t \leq T - t^*$ such that $\sigma_{t+t^*}$ contains at least two elements from $[n] \setminus M_{t+t^*}$. Recall that $|[n] \setminus M_{t+t^*}| \leq |[n] \setminus M_{t^*}| = |S_{t^*}|$. By the same argument as in Lemma \ref{lemma: Z,X gap}, we get
    \[
    \mathbf{P} \left( \mathcal{B}_T \bigg| |S_{t^*}| \leq \log n \right) \leq \frac{1}{\log n}.
    \]
Recall that the processes \(Y_t\) and \(\rho_t\) are constructed on the same probability space, using the same \(Q_{n,t}\)-moves. Let \(\mathcal{B}_t\) be the bad event that the coupling between \(Y\) and \(\rho\) has failed by time \(t\). Then, on the event \(\mathcal{B}_{\tau_2}^c\), the two processes agree up to the marking time, and in particular \(Y_{\tau_2}=\rho_{\tau_3}\). Hence
$$
d_{\mathrm{TV}}\bigl(\mathcal L(Y_{\tau_2}),\mathcal L(\rho_{\tau_3})\bigr)
\leq
\mathbf P(Y_{\tau_2}\neq \rho_{\tau_3})
\leq
\mathbf P(\mathcal{B}_{\tau_2}).
$$
It remains to bound the probability of this bad event. We split according to whether the marking time \(\tau_2\) occurs before the deterministic time \(T\), and according to whether the number of initially unmarked cards is at most \(\log n\). Since \(\mathcal{B}_{\tau_2}\subseteq \{\tau_2\geq T\}\cup \mathcal{B}_T\), we have
$$
\begin{aligned}
    \mathbf P(\mathcal{B}_{\tau_2})
    &\leq
    \mathbf P(\tau_2\geq T)+\mathbf P(\mathcal{B}_T) \\
    &\leq
    \mathbf P(\tau_2\geq T)
    +\mathbf P\bigl(\mathcal{B}_T\,\big|\, |S_{t^*}|\leq \log n\bigr)
    +\mathbf P(|S_{t^*}|>\log n).
\end{aligned}
$$
By the preceding estimates, the three terms on the right-hand side satisfy
$$
\mathbf P(\tau_2\geq T)\leq n^{-1+o(1)}, \quad\mathbf P\bigl(\mathcal{B}_T\,\big|\, |S_{t^*}|\leq \log n\bigr)\leq \frac{1}{\log n},
$$
and
$$
\mathbf P(|S_{t^*}|\geq \log n)\leq n^{-\omega(1)}.
$$
Combining these estimates gives
$$
d_{\mathrm{TV}}\bigl(\mathcal L(Y_{\tau_2}),\mathcal L(\rho_{\tau_3})\bigr)\leq\frac{1}{\log n}+n^{-1+o(1)}+n^{-\omega(1)}\leq\frac{1+o(1)}{\log n}.
$$
This completes the proof.
\end{proof}

It remains to show that, when $k$ is even, the parity of the genuine hitting time $\tau$ is asymptotically balanced.

\begin{lemma}\label{lem:parity of the hitting time}
Suppose that \(k\) is even and \(2\leq k=o(n/(\log n)^4)\). Then 
$$ 
\mathbf P(\tau \text{ is even})=\frac12+O\left(\epsilon_n\right),
\qquad
\mathbf P(\tau \text{ is odd})=\frac12+O\left(\epsilon_n\right),
$$
where $\epsilon_n = \exp\left(-(\log \log n)^{1/2+o(1)}\right).$
In particular,
$$
d_{\mathrm{TV}}\left(
\mathbf P(\tau \text{ is even})U_{\mathfrak A_n}+\mathbf P(\tau \text{ is odd})U_{\mathfrak A_n^c},
U_{\mathfrak S_n}\right)=\exp\left(-(\log \log n)^{1/2+o(1)}\right).
$$
\end{lemma}

\begin{proof}
We first separate the exceptional event $\{\tau \leq t^*\}$. By \eqref{eq: finish early}, we already have $\mathbf{P}(\tau \leq t^*) \leq \epsilon_n.$ Let \(S_t\) be the set of cards not touched by time \(t\). By \eqref{eq: St small}, we  have that $\mathbf P(|S_{t^*}|>\log n)=n^{-\omega(1)}$.
Using the same argument as in Lemma \ref{lemma: Z,X gap}, we get that $$\mathbf{P}\left(
\mathcal B \bigg| |S_{t^*}|\leq \log n
\right)
\leq
O\left(\frac{k(\log n)^3}{n}\right)
=
o\left(\frac1{\log n}\right),$$
where $\mathcal B$ is the event that  at some step after $t^*$, the chosen $k$-cycle touches at least two elements which are still untouched at that time. So after time \(t^*\), with high probability the number of untouched cards decreases by one whenever it decreases. 
Therefore, with probability \(1-O(\epsilon_n)\), the process reaches a time
$$
\sigma:=\inf\{t\geq t^*: |S_t|=1\}
$$
before the hitting time \(\tau\).

On the event that \(\sigma<\tau\), let \(i\) be the unique card not touched by time \(\sigma\). After time \(\sigma\), the hitting time \(\tau\) is obtained by waiting until a sampled \(k\)-cycle contains \(i\). At each subsequent step this happens with probability \(k/n\), independently of the past. Hence
$\tau-\sigma$ has the geometric distribution on \(\{1,2,\ldots\}\) with success probability \(k/n\). Therefore,
$$
\left|\mathbf P(\tau-\sigma\text{ is even})-\mathbf P(\tau-\sigma\text{ is odd})\right|
=
\frac{k/n}{2-k/n}
=
O\left(\frac{k}{n}\right).
$$
Consequently, conditional on the history up to time \(\sigma\), the parity of \(\tau\) has bias at most \(O(k/n)\). Combining this with the exceptional probabilities above gives
\begin{equation}
    \begin{aligned}
        \left|
\mathbf P(\tau \text{ is even})-\frac12
\right|
&\leq
O\left(\frac{k}{n}\right)
+
\mathbf P(\tau\leq t^*)
+
\mathbf P(|S_{t^*}|>\log n)
+
\mathbf P\left(
\mathcal B\cap\{|S_{t^*}|\leq\log n\}
\right)=
O\left(\epsilon_n\right).
    \end{aligned}
\end{equation}

The same bound for \(\mathbf P(\tau\text{ is odd})\) follows since the two probabilities sum to one.

Finally, since $U_{\mathfrak S_n}=\frac12 U_{\mathfrak A_n}
+\frac12 U_{\mathfrak A_n^c}$, we have
\begin{equation}
d_{\mathrm{TV}}\left(
\mathbf P(\tau \text{ is even})U_{\mathfrak A_n}+\mathbf P(\tau \text{ is odd})U_{\mathfrak A_n^c},U_{\mathfrak S_n}\right)
=\left|\mathbf P(\tau \text{ is even})-\frac12\right|=\exp\left(-(\log \log n)^{1/2+o(1)}\right). 
\end{equation}
This completes the proof.
\end{proof}

\begin{proof}[Proof of \Cref{thm: hitting time mixing_main thm}]
We combine the uniformity of the auxiliary process \(\rho_t\) with the comparison estimates between the three auxiliary chains and the original chain.

First consider the case where \(k\) is odd. By the uniformity statement proved above, conditional on the marked set \(M_t\), the random permutation \(\rho_t\) is uniform on \(\Omega_t(S_t)\). At the stopping time \(\tau_3\), all elements have been marked, so \(S_{\tau_3}=\emptyset\). Since \(k\) is odd, the walk stays in \(\mathfrak A_n\), and therefore
$$
\rho_{\tau_3}\sim U_{\mathfrak A_n}.
$$
Hence, by the triangle inequality,
$$
d_{\mathrm{TV}}\bigl(\mathcal L(X_\tau),U_{\mathfrak A_n}\bigr)\leq
d_{\mathrm{TV}}\bigl(\mathcal L(X_\tau),\mathcal L(Z_{\tau_1})\bigr)
+d_{\mathrm{TV}}\bigl(\mathcal L(Z_{\tau_1}),\mathcal L(Y_{\tau_2})\bigr) 
+d_{\mathrm{TV}}\bigl(\mathcal L(Y_{\tau_2}),\mathcal L(\rho_{\tau_3})\bigr).
$$
By the estimates proved in the preceding lemmas,
$$
d_{\mathrm{TV}}\bigl(\mathcal L(X_\tau),\mathcal L(Z_{\tau_1})\bigr)
\leq
\exp\left(-(\log\log n)^{1/2+o(1)}\right),\,\,
d_{\mathrm{TV}}\bigl(\mathcal L(Z_{\tau_1}),\mathcal L(Y_{\tau_2})\bigr)
\leq
n^{-1+o(1)},
$$
and
$$
d_{\mathrm{TV}}\bigl(\mathcal L(Y_{\tau_2}),\mathcal L(\rho_{\tau_3})\bigr)
\leq
\frac{1+o(1)}{\log n}.
$$
Since
$$
n^{-1+o(1)}+\frac{1+o(1)}{\log n}
\leq
\exp\left(-(\log\log n)^{1/2+o(1)}\right),
$$
we conclude that
$$
d_{\mathrm{TV}}\bigl(\mathcal L(X_\tau),U_{\mathfrak A_n}\bigr)
\leq
\exp\left(-(\log\log n)^{1/2+o(1)}\right).
$$

It remains to treat the case where \(k\) is even. In this case each \(k\)-cycle is an odd permutation, so the parity of the walk alternates with time. Again, the uniformity of \(\rho_t\) conditional on the marked set implies that, at the marking time \(\tau_3\),
$$
\rho_{\tau_3}
\sim
\mathbf P(\tau_3\text{ is even})U_{\mathfrak A_n}
+
\mathbf P(\tau_3\text{ is odd})U_{\mathfrak A_n^c}.
$$
Since \Cref{lem: Y and rho} gives \(\tau_2=\tau_3\) under the coupling, and the previous comparison between \(Z_{\tau_1}\) and \(X_\tau\) includes the event \(\tau_1\neq \tau\), the same triangle inequality gives
\begin{multline}
d_{\mathrm{TV}}\left(
\mathcal L(X_\tau),
\mathbf P(\tau\text{ is even})U_{\mathfrak A_n}+
\mathbf P(\tau\text{ is odd})U_{\mathfrak A_n^c}
\right) \\\leq d_{\mathrm{TV}}\bigl(\mathcal L(X_\tau),\mathcal L(Z_{\tau_1})\bigr)
+d_{\mathrm{TV}}\bigl(\mathcal L(Z_{\tau_1}),\mathcal L(Y_{\tau_2})\bigr)
+d_{\mathrm{TV}}\left(\mathcal L(Y_{\tau_2}),
\mathbf P(\tau\text{ is even})U_{\mathfrak A_n}+\mathbf P(\tau\text{ is odd})U_{\mathfrak A_n^c}
\right).
\end{multline}
The last term is bounded by
$$
d_{\mathrm{TV}}\bigl(\mathcal L(Y_{\tau_2}),\mathcal L(\rho_{\tau_3})\bigr)
+
d_{\mathrm{TV}}\left(
\mathcal L(\rho_{\tau_3}),
\mathbf P(\tau\text{ is even})U_{\mathfrak A_n}
+
\mathbf P(\tau\text{ is odd})U_{\mathfrak A_n^c}
\right).
$$
The first part is at most \((1+o(1))/\log n\) by \Cref{lem: Y and rho}. For the second part, note that \(\rho_{\tau_3}\) has the parity mixture determined by \(\tau_3\), while the target mixture is determined by \(\tau\). Therefore this contribution is bounded by $\mathbf P(\tau_3\neq \tau)$. Since $\tau_3=\tau_2$ under the coupling, we have
$$
\mathbf P(\tau_3\neq \tau)\leq\mathbf P(\tau_2\neq \tau_1)+\mathbf P(\tau_1\neq \tau).
$$
By the comparison estimate between $\tau_1$ and $\tau_2$ in \Cref{lem: Z and Y}, together
with the bound
$$
\mathbf P(\tau_1\neq \tau)\leq\exp\left(-(\log\log n)^{1/2+o(1)}\right),
$$
we obtain
$$
\mathbf P(\tau_3\neq \tau)\leq n^{-1+o(1)}+\exp\left(-(\log\log n)^{1/2+o(1)}\right)=\exp\left(-(\log\log n)^{1/2+o(1)}\right).
$$
Thus
$$
d_{\mathrm{TV}}\left(
\mathcal L(X_\tau),
\mathbf P(\tau\text{ is even})U_{\mathfrak A_n}
+
\mathbf P(\tau\text{ is odd})U_{\mathfrak A_n^c}
\right)
\leq
\exp\left(-(\log\log n)^{1/2+o(1)}\right).
$$
Finally, by \Cref{lem:parity of the hitting time},
$$
d_{\mathrm{TV}}\left(
\mathbf P(\tau\text{ is even})U_{\mathfrak A_n}+\mathbf P(\tau\text{ is odd})U_{\mathfrak A_n^c},
U_{\mathfrak S_n}
\right)=\exp\left(-(\log\log n)^{1/2+o(1)}\right).
$$
Therefore, another application of the triangle inequality gives
$$
d_{\mathrm{TV}}\bigl(\mathcal L(X_\tau),U_{\mathfrak S_n}\bigr)
\leq
\exp\left(-(\log\log n)^{1/2+o(1)}\right).
$$
This completes the proof.
\end{proof}

\section{Hitting time mixing for large cycles}\label{sec: large cycles}

In this section, we prove \Cref{thm: large cycle}. Unless otherwise noted, we use the notations from \Cref{sec: Jain-Sawhney}. As we will see, the proof of \Cref{thm: large cycle} reduces to a Fourier estimate for the two-step distribution $P_n^{*2}$. The following proposition gives the required bound on the contribution from all nontrivial irreducible representations.

\begin{prop}\label{prop: fourier bound large cycle}
Suppose $n=\omega(1),1\le n-k\le o(n^{1/2})$.  Then, we have 
$$\sum_{\lambda\ne(n),(1^n)}\frac{\chi_\l(\tau_k)^4}{d_\lambda^2}=o(1).$$
\end{prop}

\begin{proof}[Proof of \Cref{thm: large cycle}, assuming \Cref{prop: fourier bound large cycle}]
After the first random cycle, the probability that we stop at the next cycle is exactly the probability that the second $k$-cycle covers the remaining $n-k$ cards. Therefore, we have
$$\mathbf{P}(\tau=2)=\frac{\binom{k}{n-k}}{\binom{n}{n-k}}=\left(1-\frac{n-k}{n}\right)\cdots\left(1-\frac{n-k}{k+1}\right)=1-O\left(\frac{(n-k)^2}{n}\right)=1-o(1).$$
Hence, it suffices to prove that $d_{TV}(P_n^{*2},U_{\A_n})=o(1)$. Recall from \eqref{eq: Fourier transform of Pn} that 
$$\widehat P_n^{*2}(\lambda)=(\widehat P_n(\lambda))^2=\frac{\chi_\lambda(\tau_k)^2}{d_\lambda^2}.$$
Moreover, the Fourier transform of $U_{\A_n}$ is one over the partitions $(n),(1^n)$, and zero on the other partitions. Now, applying the Cauchy-Schwarz inequality followed by the Plancherel formula, we have that
\begin{align}
\begin{split}
4\cdot d_{TV}(P_n^{*2},U_{\A_n})^2
&\le\sum_{\l\vdash n}d_\l\Tr\left((\widehat P_n^{*2}(\lambda)-\widehat U_{\A_n}(\l))(\widehat P_n^{*2}(\lambda)-\widehat U_{\A_n}(\l))^\dag\right)\\
&=\sum_{\lambda\ne(n),(1^n)}\frac{\chi_\l(\tau_k)^4}{d_\lambda^2}\\
&=o(1).\\
\end{split}
\end{align}
Therefore, we conclude $d_{TV}(P_n^{*2},U_{\A_n})=o(1)$, which completes the proof.
\end{proof}

\begin{rmk}
The assumption $k=n-o(n^{1/2})$ is sharp for the estimate in \Cref{prop: fourier bound large cycle}. Indeed, consider the partition $\lambda=(n-1,1)$. In this case,
$$
\chi_{(n-1,1)}(\tau_k)=d_{(n-k-1,1)}=n-k-1,
\qquad d_{(n-1,1)}=n-1.
$$
Therefore the single term corresponding to $\lambda=(n-1,1)$ in the spectral sum is
$$
\frac{\chi_{(n-1,1)}(\tau_k)^4}{d_{(n-1,1)}^2}=\frac{(n-k-1)^4}{(n-1)^2}.
$$
If $k=n-\Omega(n^{1/2})$, then this term is
already $\Omega(1)$. Thus the conclusion of the proposition cannot hold in that range, showing that the condition $k=n-o(n^{1/2})$ is optimal.
\end{rmk}

For the rest of this section, we aim to prove \Cref{prop: fourier bound large cycle}. We first record a consequence of the (recursive) Murnaghan-Nakayama rule for a single
large cycle. The uniqueness statement in the following proposition is the same observation as Lulov-Pak \cite[Lemma 5.1]{lulov2002rapidly}. For our later estimates, we also keep track of the resulting shape of this unique rim hook.

\begin{lemma}\label{lem: MN one term three types}
Let $n\ge 1$, let $k$ satisfy $1\le n-k<n/2$, and write $r:=n-k$.
Then for every partition $\lambda\vdash n$, 
there exists at most one $\tilde\lambda\vdash r$ such that $\lambda/\widetilde\lambda$ is a rim hook (also called border strip) of size $k$, and in this case we have 
$|\chi_\lambda(\tau_k)|=d_{\widetilde\lambda}$. Otherwise, we have $\chi_\lambda(\tau_k)=0$.
\end{lemma}

\begin{proof}
By the Murnaghan-Nakayama rule applied to the cycle type $(k,1^{n-k})$, we have
$$
\chi_\lambda(\tau_k)=\sum_{\substack{\widetilde\lambda\vdash n-k\\
\lambda/\widetilde\lambda\text{ is a rim hook of size }k}}
(-1)^{\operatorname{ht}(\lambda/\widetilde\lambda)-1}
d_{\widetilde\lambda}.
$$
Since $k>n/2$, Lulov-Pak \cite[Lemma 5.1]{lulov2002rapidly} shows that there is at most one way to remove a rim hook of size $k$ from $\lambda$. Hence the above sum has at most one term, and therefore 
$$|\chi_\lambda(\tau_k)|=d_{\widetilde\lambda}$$
when such a rim hook exists, while $\chi_\lambda(\tau_k)=0$ otherwise.
\end{proof}

\begin{prop}
\label{prop: fixed core rim hook sum}
Suppose that $r=o(n^{1/2})$, and fix a partition
$\widetilde\lambda\vdash r$ satisfying
$\widetilde\lambda_1'\le \widetilde\lambda_1\le r-1$. Then
$$
\sum_{\substack{\lambda\vdash n:\ \widetilde\lambda\subseteq\lambda\\
\lambda/\widetilde\lambda\text{ is a rim hook of size }n-r}}
\left(\frac{d_{\widetilde\lambda}}{d_\lambda}\right)^2
=
O\left(
\left(\frac{3r}{n}\right)^{2(r-\widetilde\lambda_1)}
\right).
$$
\end{prop}

\begin{proof}
Since $\lambda/\widetilde\lambda$ is a rim hook of size $n-r$, it is
contained in the outer rim of $\lambda$, and hence
$$
n-r\le \lambda_1+\lambda_1'-1.
$$
Since $r=o(n^{1/2})$, every partition appearing in the sum satisfies, for all sufficiently large $n$, either
$\lambda_1\ge \frac n3$ or $\lambda_1'\ge \frac n3$. We first consider the contribution of the partitions satisfying $\lambda_1\ge n/3$.

We will repeatedly use the following consequence of the hook-length
formula. Given a partition $\eta\vdash N$, we will write $\eta^*$ for the partition obtained from $\eta$ by deleting its first row, and we will Denote $s:=|\eta^*|$. Then
$$
\frac{d_\eta}{d_{\eta^*}}
=
\binom{N}{s}
\prod_{j=1}^{\eta_2}
\frac{\eta_1-j+1}
{\eta_1-j+1+\eta_j'-1}.
$$
If $\eta_1\ge n/3$ and $s\le r+1$, then $r=o(n^{1/2})$ implies
$$
\prod_{j=1}^{\eta_2}
\frac{\eta_1-j+1}
{\eta_1-j+1+\eta_j'-1}
=1-o(1),
$$
uniformly over all such $\eta$. Consequently,
$d_\eta\ge(1-o(1))\binom{N}{s}d_{\eta^*}$.
We will also use the upper bound from
\Cref{lem: dimension bound}, which claims that
$$
d_{\widetilde\lambda}
\le
\binom{r}{r-\widetilde\lambda_1}
d_{\widetilde\lambda^*}.
$$

To begin with, consider the partition $\lambda$ for which $\lambda/\widetilde\lambda$ is a horizontal strip added entirely to the first row. In this case,
$$
\lambda^*=\widetilde\lambda^*,
\qquad
|\lambda^*|=r-\widetilde\lambda_1.
$$
Applying the hook-length comparison to $\lambda$, together with
\Cref{lem: dimension bound} for $\widetilde\lambda$, gives
$$\frac{d_{\widetilde\lambda}}{d_\lambda}\le\frac{
\binom{r}{r-\widetilde\lambda_1}
d_{\widetilde\lambda^*}
}{(1-o(1))\binom{n}{r-\widetilde\lambda_1}
d_{\lambda^*}}=O\left(
\left(\frac{r}{n}\right)^{r-\widetilde\lambda_1}
\right).$$

For every other partition $\lambda$ with $\lambda_1\ge n/3$, trim the
rim hook immediately after its first downward turn, and denote the
remaining partition by $\mu$. Then $\mu\subseteq\lambda$, and the
diagram $\mu^*$ obtained by deleting the first row of $\mu$ contains
$\widetilde\lambda^*$ together with at least one additional box. In
particular,
$$
r-\widetilde\lambda_1+1\le |\mu^*|\le r.
$$
By the branching rule, $d_{\mu^*}\ge d_{\widetilde\lambda^*}$. Therefore, the hook-length formula and
\Cref{lem: dimension bound} give
$$\frac{d_{\widetilde\lambda}}{d_\lambda}\le\frac{d_{\widetilde\lambda}}{d_\mu}\le\frac{\binom{r}{r-\widetilde\lambda_1}d_{\widetilde\lambda^*}
}{(1-o(1))\binom{|\mu|}{|\mu^*|}d_{\mu^*}}\le(1+o(1))\frac{\binom{r}{r-\widetilde\lambda_1}}{\binom{|\mu|}{r-\widetilde\lambda_1+1}}=O\left(\left(\frac{3r}{n}\right)^{r-\widetilde\lambda_1+1}\right).$$
Here, we use the assumption that $|\mu|\ge \lambda_1\ge n/3$.

There are only $O(n)$ possible partitions $\lambda$ of this second
kind. Indeed, with $\widetilde\lambda$ fixed, one may imagine sliding
the rim hook along the outer boundary, starting with its horizontal
part in the top row and moving its turning point successively
downward until its vertical part reaches the bottom of the first
column. The position of this turning point determines $\lambda$, and
there are at most $O(n)$ possible positions.

It follows that
\begin{align*}
\sum_{\substack{\lambda\supseteq\widetilde\lambda\\
\lambda/\widetilde\lambda\text{ is a rim hook of size }n-r\\
\lambda_1\ge n/3}}
\left(\frac{d_{\widetilde\lambda}}{d_\lambda}\right)^2 &\le O\left(\left(\frac{r}{n}\right)^{
2(r-\widetilde\lambda_1)}\right)+O(n)
O\left(\left(\frac{3r}{n}\right)^{
2(r-\widetilde\lambda_1+1)}\right) \\
&=O\left(\left(\frac{3r}{n}\right)^{
2(r-\widetilde\lambda_1)}\right),
\end{align*}
where the last step follows from $r=o(n^{1/2})$.

Finally, consider the partitions satisfying $\lambda_1'\ge n/3$.
Applying the preceding argument after conjugation gives the bound
$$
O\left(
\left(\frac{3r}{n}\right)^{
2(r-\widetilde\lambda_1')}
\right).
$$
Since $\widetilde\lambda_1\ge\widetilde\lambda_1'$, we have $r-\widetilde\lambda_1'\ge r-\widetilde\lambda_1$. As $3r/n<1$ for all sufficiently large $n$, this contribution is also
$$
O\left(
\left(\frac{3r}{n}\right)^{
2(r-\widetilde\lambda_1)}
\right).
$$
Combining the two cases completes the proof.
\end{proof}

\begin{prop}
\label{prop: one row and one column cores}
Suppose that $r=o(n^{1/2})$. Then
$$
\sum_{\substack{\lambda\vdash n,\ \lambda\neq(n),(1^n)\\
\lambda/(r)\text{ is a rim hook of size }n-r}}
\frac{1}{d_\lambda^2}=O\left(\frac{r^2}{n}\right)=o(1).
$$
Similarly,
$$
\sum_{\substack{\lambda\vdash n,\ \lambda\neq(n),(1^n)\\
\lambda/(1^r)\text{ is a rim hook of size }n-r}}\frac{1}{d_\lambda^2}=O\left(\frac{r^2}{n}\right)=o(1).
$$
\end{prop}

\begin{proof}
The first estimate follows routinely from the proof of \Cref{prop: fixed core rim hook sum}. Indeed, when $\widetilde\lambda=(r)$, the exceptional horizontal extension is precisely $\lambda=(n)$, which is excluded from the sum. Every remaining partition belongs to the second case considered there. There are at most $O(n)$ such partitions, and each of them satisfies
$$
\frac{d_{(r)}}{d_\lambda}=\frac{1}{d_\lambda}
=O\left(\frac{2r}{n}\right).
$$
Consequently,
$$
\sum_{\substack{\lambda\vdash n,\ \lambda\neq(n),(1^n)\\
\lambda/(r)\text{ is a rim hook of size }n-r}}
\frac{1}{d_\lambda^2}
\le O(n)O\left(\frac{r^2}{n^2}\right)=O\left(\frac{r^2}{n}\right)=o(1).
$$

The second estimate follows from the first one by conjugation, using
$d_{\lambda'}=d_\lambda$.
\end{proof}

Finally, we are prepared for our proof of \Cref{prop: fourier bound large cycle}.

\begin{proof}[Proof of \Cref{prop: fourier bound large cycle}]
Write $r:=n-k$. By \Cref{lem: MN one term three types}, for every
$\lambda\vdash n$, either $\chi_\lambda(\tau_k)=0$, or there exists a
unique partition $\widetilde\lambda\vdash r$ such that
$\lambda/\widetilde\lambda$ is a rim hook of size $k$, in which case
$$
|\chi_\lambda(\tau_k)|=d_{\widetilde\lambda}.
$$
Therefore,
\begin{align*}
\sum_{\lambda\neq(n),(1^n)}
\frac{\chi_\lambda(\tau_k)^4}{d_\lambda^2}
&=
\sum_{\widetilde\lambda\vdash r}
d_{\widetilde\lambda}^2
\sum_{\substack{\lambda\vdash n:\,
\lambda/\widetilde\lambda\text{ is a rim hook of size }k\\
\lambda\neq(n),(1^n)}}
\left(\frac{d_{\widetilde\lambda}}{d_\lambda}\right)^2.
\end{align*}

Since conjugation preserves dimensions and sends the pair
$(\lambda,\widetilde\lambda)$ to
$(\lambda',\widetilde\lambda')$, it is enough, up to a factor of $2$,
to sum over the partitions satisfying
$\widetilde\lambda_1\ge\widetilde\lambda_1'$.
The contribution of $\widetilde\lambda=(r)$ is
$o(1)$ by \Cref{prop: one row and one column cores}. It remains to consider
$\widetilde\lambda\neq(r)$. For
$1\le l\le r-1$, \Cref{lem: dimension bound} gives
$$
d_{\widetilde\lambda}
\le
\binom{r}{r-l}d_{\widetilde\lambda^*}
$$
whenever $\widetilde\lambda_1=l$, where
$\widetilde\lambda^*$ is obtained by deleting the first row. Consequently, using the identity
$$
\sum_{\mu\vdash r-l}d_\mu^2=(r-l)!,
$$
we obtain
$$d_{\widetilde\lambda}^2\le
\binom{r}{r-l}^2
\sum_{\mu\vdash r-l}d_\mu^2=
\binom{r}{r-l}^2(r-l)!\le
\frac{r^{2(r-l)}}{(r-l)!}.$$

Applying \Cref{prop: fixed core rim hook sum} and then summing over
$l$ therefore gives
\begin{align*}
\sum_{\substack{\widetilde\lambda\vdash r,\,
\widetilde\lambda\neq(r)\\
\widetilde\lambda_1\ge\widetilde\lambda_1'}}
d_{\widetilde\lambda}^2
\sum_{\substack{\lambda\vdash n:\,
\lambda/\widetilde\lambda\text{ is a rim hook of size }k}}
\left(\frac{d_{\widetilde\lambda}}{d_\lambda}\right)^2 &\le
\sum_{l=1}^{r-1}
O\left(\left(\frac{3r}{n}\right)^{2(r-l)}
\right)\sum_{\substack{\widetilde\lambda\vdash r\\\widetilde\lambda_1=l}}
d_{\widetilde\lambda}^2 \\
&\le
\sum_{l=1}^{r-1}
O\left(
\frac{(3r^2/n)^{2(r-l)}}{(r-l)!}
\right) \\
&\le
O\left(
\exp\left(\left(\frac{3r^2}{n}\right)^2\right)-1
\right)\\
&=o(1),
\end{align*}
where the last equality follows from $r=o(n^{1/2})$.

The conjugate partitions contribute the same amount. Combining the
preceding estimates proves
$$
\sum_{\lambda\neq(n),(1^n)}
\frac{\chi_\lambda(\tau_k)^4}{d_\lambda^2}
=o(1).
$$
\end{proof}

\bibliographystyle{plain}
\bibliography{references.bib}

\end{document}